\documentclass[a4paper,12pt]{amsart}

\usepackage{environ}
\usepackage{fancyhdr}
\usepackage{mabliautoref}
\usepackage{amssymb,amsthm,amsmath}
\RequirePackage[dvipsnames,usenames]{xcolor}
\usepackage{hyperref}
\usepackage{cleveref}
\usepackage{mathtools}
\usepackage{tcolorbox}
\usepackage[all]{xy}
\usepackage{tikz}
\usepackage{tikz-cd}

\makeatletter
\def\fixtikzforbreqn#1#2{%
	\protected\edef#1{\noexpand\ifmmode\mathchar\the\mathcode`#2 \noexpand\else#2\noexpand\fi}%
}
\fixtikzforbreqn\tikz@nonactivesemicolon;
\fixtikzforbreqn\tikz@nonactivecolon:
\fixtikzforbreqn\tikz@nonactivebar|
\fixtikzforbreqn\tikz@nonactiveexlmark!
\makeatother

\usetikzlibrary{decorations.markings,shapes,positioning}
\usepackage{enumitem}
\usepackage{stmaryrd}
\usepackage{xfrac}
\usepackage{fix-cm}
\usepackage{faktor}
\usepackage{minibox}
\usepackage{commath}
\usepackage{chngcntr}
\usepackage{setspace}
\usepackage{array}
\usepackage[toc,page]{appendix}
\usepackage{calligra,mathrsfs}
\usepackage{mathtools}
\usepackage{bbm}
\hypersetup{
bookmarks,
bookmarksdepth=3,
bookmarksopen,
bookmarksnumbered,
pdfstartview=FitH,
colorlinks,backref,hyperindex,
linkcolor=Sepia,
anchorcolor=BurntOrange,
citecolor=Cyan,
filecolor=BlueViolet,
menucolor=Yellow,
urlcolor=OliveGreen
}

\newtcolorbox{activitybox}[1][]{%
	breakable,
	enhanced,
	colback=lightergray,
	boxrule=3pt,
	arc=5pt,
	outer arc=5pt,
	boxsep=10pt,
	colframe=darkergray,
	coltitle=white,
	#1
}

\newcommand{\quot}[2]{%
	\raise1ex\hbox{$#1$}\Big/\lower1ex\hbox{$#2$}%
}

\newcommand{\colim}{\varinjlim}
\renewcommand{\lim}{\varprojlim}

\makeatletter\def\Rvarlim@#1#2{%
	\vtop{\m@th\ialign{##\cr
			\hfil$#1\operator@font Rlim$\hfil\cr
			\noalign{\nointerlineskip\kern1.5\ex@}#2\cr
			\noalign{\nointerlineskip\kern-\ex@}\cr}}%
}
\makeatother

\makeatletter\def\Rlim{%
	\mathop{\mathpalette\Rvarlim@{\leftarrowfill@\textstyle}}\nmlimits@
}
\makeatother

\newcommand{\expl}[2]{\underset{\mathclap{\minibox[c]{$\uparrow$\\ \fbox{\footnotesize #2}}}}{#1}}

\def\xto#1#2{\xrightarrow{#1}{#2}}
\def\xto#1{\xrightarrow{#1}}
\newcommand{\surj}{\twoheadrightarrow}
\newcommand{\inj}{\hookrightarrow}

\newcommand{\HHom}{\mathcal{H}om}
\newcommand{\EExt}{\mathcal{E}xt}
\newcommand{\et}{\acute{e}t}

\newcommand{\stacksproj}[1]{{\cite[Tag~\href{http://stacks.math.columbia.edu/tag/#1}{#1}]{Stacks_Project}}}

\newcommand{\inc}{\subseteq}

\DeclareMathAlphabet{\mathchanc}{OT1}{pzc}%
                                 {m}{it}

\newcommand{\bunit}{\mathbbm{1}}

\newcommand{\bA}{\mathbb{A}}

\newcommand{\bD}{\mathbb{D}}

\newcommand{\bF}{\mathbb{F}}

\newcommand{\bQ}{\mathbb{Q}}
\newcommand{\bR}{\mathbb{R}}

\newcommand{\bZ}{\mathbb{Z}}

\newcommand{\scr}{\mathcal}
\newcommand{\cA}{\scr{A}}

\newcommand{\cC}{\scr{C}}
\newcommand{\cD}{\scr{D}}

\newcommand{\cF}{\scr{F}}
\newcommand{\cG}{\scr{G}}
\newcommand{\cH}{\scr{H}}
\newcommand{\cI}{\scr{I}}

\newcommand{\cM}{\scr{M}}
\newcommand{\cN}{\scr{N}}
\newcommand{\cO}{\scr{O}}

\newcommand{\cQ}{\scr{Q}}
\newcommand{\cR}{\scr{R}}

\newcommand{\cT}{\scr{T}}

\DeclareMathOperator{\Tr}{Tr}
\DeclareMathOperator{\codim}{codim}
\DeclareMathOperator{\coker}{{coker}}

\DeclareMathOperator{\Ext}{Ext}

\DeclareMathOperator{\Frac}{Frac}

\DeclareMathOperator{\Hom}{Hom}

\DeclareMathOperator{\im}{{im}}

\DeclareMathOperator{\red}{red}

\DeclareMathOperator{\Spec}{{Spec}}

\DeclareMathOperator{\Supp}{{Supp}}

\DeclareMathOperator{\Cone}{{Cone}}

\DeclareMathOperator{\Crys}{Crys}

\DeclareMathOperator{\RGamma}{R\Gamma}

\DeclareMathOperator{\Mod}{Mod}
\DeclareMathOperator{\hol}{hol}

\DeclareMathOperator{\perf}{perf}

\DeclareMathOperator{\QCoh}{QCoh}
\DeclareMathOperator{\QCrys}{QCrys}

\DeclareMathOperator{\Coh}{Coh}
\DeclareMathOperator{\Sol}{Sol}

\DeclareMathOperator{\rig}{rig}
\DeclareMathOperator{\Spf}{Spf}

\DeclareMathOperator{\ShHom}{\mathscr{H}\text{\kern -3pt {\calligra\large om}}\,}

\DeclareFontFamily{OT1}{pzc}{}
\DeclareFontShape{OT1}{pzc}{m}{it}{<-> s * [1.200] pzcmi7t}{}
\DeclareMathAlphabet{\mathpzc}{OT1}{pzc}{m}{it}

\AtBeginDocument{%
	\mathchardef\phialt=\phi
	\mathchardef\phi=\varphi
}

\newcommand{\factor}[2]{\left. \raise 2pt\hbox{\ensuremath{#1}} \right/
        \hskip -2pt\raise -2pt\hbox{\ensuremath{#2}}}

\makeatletter
\renewcommand\subsection{
  \renewcommand{\sfdefault}{pag}
  \@startsection{subsection}%
  {2}{0pt}{.8\baselineskip}{.4\baselineskip}{\raggedright
    \sffamily\itshape\small\bfseries
  }}
\renewcommand\section{
  \renewcommand{\sfdefault}{phv}
  \@startsection{section} %
  {1}{0pt}{\baselineskip}{.8\baselineskip}{\centering
    \sffamily
    \scshape
    \bfseries
}}
\makeatother

\definecolor{gr}{rgb}{0,0.5,0}

\newcommand{\Addresses}{{% additional braces for segregating \footnotesize
		\bigskip
		\footnotesize
		
		\textsc{\'Ecole Polytechnique F\'ed\'erale de Lausanne, SB MATH CAG, MA C3 615 (B\^atiment MA), Station 8, CH-1015 Lausanne, Switzerland}\par\nopagebreak
		\textit{E-mail address}: \texttt{jefferson.baudin@epfl.ch}

}}

\author{Jefferson Baudin}
\date{}

\usepackage[margin=1.0in]{geometry}
\setlist{  
	listparindent=\parindent,
	parsep=0pt,
}


\title[Grauert--Riemenschndeider vanishing for Witt canonical sheaves]{A Grauert--Riemenschneider vanishing theorem for Witt canonical sheaves}

\begin{document}
	\maketitle
	\begin{abstract}
		We prove a Witt vector version of the usual Grauert--Riemenschneider vanishing theorem over perfect fields of positive characteristic, solving a question raised by Blickle, Esnault, Chatzistamatiou and Rülling. We then deduce some rationality consequences for $F$--rational singularities.
	\end{abstract}

	\tableofcontents

\section{Introduction}

\subsection{Main results}

It is now well--known that many vanishing theorems that hold in characteristic zero fail in positive characteristic. Perhaps the most common example is Kodaira vanishing \cite{Raynaud_Failure_Kodaira_vanishing}. A natural question that one can come up with is then: is there some sort of Kodaira--type vanishing that holds in positive characteristic? There are two main ways to tackle this question:

\begin{enumerate}
	\item Putting restrictions on varieties, e.g. asking that the characteristic of the base field is big enough, working in low dimensions, imposing some ordinarity conditions or focusing on certain special types of varieties (see e.g. \cite{Deligne_Illusie_Relevements_modulo_p_2_et_decomposition_du_complexe_de_de_Rham, Mehta_Ramanathan_Frobenius_splitting_and_cohomology_vanishing_for_Schubert_varieties, Tanaka_KVV_for_toric_varieties,  Gongyo_Takagi_Kollar_injectivity_for_globally_F_regular_varieties, Ekedahl_Canonical_models_of_surfaces_of_general_type_in_positive_characteristic, Di_Cerbo_Fanelli_Effective_Matsusaka_theorem_for_surfaces_in_char_p, Kawakami_On_the_KVV_for_log_CY_in_large_char, Bernasconi_Martin_Bounding_geom_integral_del_Pezzo_surfaces, Arvidsson_Bernasconi_Lacini_KVV_for_log_dP_surfaces_in_pos_char, Kawakami_Tanaka_Vanishing_Theorems_for_Fano_threefolds_in_positive_characteristic, Bernasconi_Kollar_vanishing_theorems_for_threefolds_in_char_>5,  Baudin_Bernasconi_Kawakami_Frobenius_GR_fails}).
	
	\item Working in greater generality (maybe even any characteristic and dimension), but weakening the vanishing statement we are looking for (see e.g. \cite{Esnault_Varieties_over_a_finite_field_with_trivial_Chow_group_of_0_cycles_have_a_rational_point, Berthelot_Bloch_Esnault_On_Witt_vector_cohomology_for_singular_varieties, Gongyo_Nakamura_Tanaka_Rational_points_on_log_fano_threefolds_over_a_finite_field, Tanaka_Vanishing_theorems_of_Kodaira_type_for_Witt_canonical_sheaves, Bhatt_Derived_Splinters_In_Positive_Characteristic, Bhatt&Co_Applications_of_perverse_sheaves_in_commutative_algebra, Baudin_Duality_between_perverse_sheaves_and_Cartier_crystals, Hacon_Pat_GV_Characterization_Ordinary_AV, Hacon_Pat_GV_Geom_Theta_Divs}). They often involve the Frobenius morphism or Witt vectors in some way.
\end{enumerate}

In this paper, we will take the second approach and try to find an analogue of \emph{Grauert--Riemenschneider vanishing}, which concerns the last piece of the de Rham complex:

\begin{thm_blank*}[\cite{Grauert_Riemenschneider_vanishing}]
	Let $\pi \colon Y \to X$ be a proper birational morphism of complex varieties, with $Y$ smooth. Then for all $i > 0$, we have \[ R^i\pi_*\omega_Y = 0, \] where $\omega_Y$ denotes the sheaf of top forms on $Y$.
\end{thm_blank*}

This theorem has many important consequences in characteristic zero, such as the rationality of klt singularities, the invariance of rational singularities under small deformations, etc... (\cite{Elkik_Rationalite_des_singularites_canoniques, Elkik_Singularites_rationnelles_et_deformations, Esnault_Viewheg_Two_dimensionaL_quotient_singularities_deform_to_quotient_singularities}). It could hence be fruitful to have a positive characteristic analogue.

As one would expect, this statement fails if we replace complex numbers by a field of positive characteristic (e.g. take a cone over a smooth projective variety violating Kodaira vanishing, see \cite[Proposition 3.13]{Hacon_Kovacs_GV_fails_in_pos_char}). A weakening that came up in \cite{Baudin_Bernasconi_Kawakami_Frobenius_GR_fails} was to ask that, even though the groups $R^i\pi_*\omega_Y$ may not vanish, they are nilpotent under the Cartier operator. Unfortunately, the main result of \emph{loc. cit.} is that even this weakening fails, already for threefolds. \\

Let us now focus on what \emph{could} hold. The Grauert--Riemenschneider vanishing theorem is analytic in nature; it follows from results in Hodge theory and their impact on the de Rham cohomology of complex varieties. The closest we are to such considerations in positive characteristic is certainly not with de Rham cohomology, which is famous for having pathological behaviour. For example, the Hodge--to--de Rham spectral sequence may not degenerate (\cite{Mumford_Pathologies_of_modular_alg_surfaces}), or the ``Betti numbers'' we obtain out of it may not be the correct ones (for example they are not invariant under smooth deformations, \cite{Kothari_Arbitrarily_large_jumps_in_de_Rham_and_Hodge_coh_of_families_in_char_p}). The right candidate seems to be its famous mixed characteristic deformation: \emph{crystalline cohomology}. There is a so--called de Rham--Witt complex, whose cohomology computes crystalline cohomology (\cite{Illusie_Complexe_de_de_Rham_Witt_et_cohomologie_cristalline}), and for example the analogue of the Hodge--to--de Rham spectral sequence degenerates after inverting $p$. Therefore, one might expect that the sheaf of Witt--top forms $W\omega_Y$ could satisfy a GR vanishing type statement (possibly after inverting $p$). This question first appeared implicitly in \cite{Blickle_Esnault_Rational_Singularities_and_rational_points}, and was then studied in \cite{Rulling_Chatzimatiaou_Hodge_Witt_cohomology_and_Witt_rational_singularities} and \cite{Tanaka_Vanishing_theorems_of_Kodaira_type_for_Witt_canonical_sheaves}. \\

The main goal of this paper is to solve it: 

\begin{thm_letter}\label{Witt-GR_intro}
	Let $\pi \colon Y \to X$ be a proper birational morphism of varieties over a perfect field, and assume that $Y$ is smooth. Then for all $i > 0$, we have that \[ R^i\pi_*W\omega_{Y, \bQ} = 0. \]
	
	Furthermore, if either $\dim(Y) \leq 3$, or $X$ is projective and has isolated singularities, then there exists $e > 0$ such that for all $i > 0$ and $n \geq 1$, we have that \[ p^e \cdot R^i\pi_*W_n\omega_Y = 0. \] 
\end{thm_letter}

\noindent The notation $(\cdot)_{\bQ}$ essentially means that we invert $p$ (see \autoref{sec_notations} for a more precise formulation). As kindly remarked by Kay Rülling, we obtain as a corollary a vanishing theorem for the étale sheaves $W_n\omega_{Y, \log}$ and the motivic complexes $\bZ/p^n(d)$, see \autoref{motivic_vanishing}.

\begin{rem_blank}
	A reader used to $W\cO$--rationality might be confused by the phrasing of the statement. For example, given any proper morphism $\pi \colon Y \to X$ and some $i > 0$, then $R^i\pi_*W\cO_{Y, \bQ}$ is zero if and only if each $R^i\pi_*W_n\cO_Y$ is $p^e$--torsion for some fixed $e > 0$ independent of $n \geq 1$ (see \cite[Proposition 3.10]{Patakfalvi_Zdanowicz_Ordinary_varieties_with_trivial_canonical_bundle_are_not_uniruled}). This is \emph{not} the case for the sheaves $W\omega_Y$ and $W_n\omega_Y$. A philosophical reason is that the Cartier operator is an isomorphism on $W\omega_Y$ (see \autoref{Witt_omega_coperfect}), while this is not at all the case for $W_n\omega_Y$, so vanishing between these objects should not relate so easily. See \autoref{name_Q_p_GR} for the precise statement that holds.
\end{rem_blank}

In characteristic zero, the main tool for proving this theorem is through the Kawamata--Viehweg vanishing theorem. We pursue the same approach in our case, so our first step is to prove the following strenghtening of \cite[Corollary 1.2]{Berthelot_Bloch_Esnault_On_Witt_vector_cohomology_for_singular_varieties}. 

\begin{thm_letter}\label{Witt_KVV_intro}
	Let $Y$ be a smooth proper variety over a perfect field, and let $D$ be an effective, big and semi--ample divisor on $Y$. Then for all $i < \dim(Y)$, we have that \[ H^i(Y, WI_{D, \bQ}) = 0. \]
\end{thm_letter}

Let us explain the notation: $D$ is effective, so $\cO_Y(-D)$ is an ideal of $\cO_Y$, which we call $I_D$. Then, we define $WI_D$ to be the ideal of $W\cO_X$ of elements whose components are in $I_D$ (the reason why we require $D$ to be effective is simply to define $I_D$). Note that this is a different sheaf than the one studied in \cite{Tanaka_Vanishing_theorems_of_Kodaira_type_for_Witt_canonical_sheaves}. As kindly suggested by Shunsuke Takagi, our proof can be generalized to show the vanishing of $H^i(Y, WI_{D, \bQ})$ with $i < \kappa(Y, D)$ for effective and semi--ample divisors $D$ on $Y$ (see \autoref{Witt_KVV}). Here, $\kappa(Y, D)$ denotes the Iitaka dimension of $D$.

We also thank Hélène Esnault for suggesting us that using arguments analogous to those of \cite{Berthelot_Esnault_Rulling_Rational_points_over_finite_fields_for_regular_modules_of_algebraic_varieties_of_Hodge_type}, one might be able to show the following: if $(Y, D)$ admits a lift to characteristic zero, then \autoref{Witt_KVV_intro} follows from the usual Kawamata--Viehweg vanishing theorem from characteristic zero.

\subsection{First applications}

\setcounter{theorem}{2}

Our applications in this paper concern singularity theory. In characteristic zero, it follows from Grauert--Riemenschneider vanishing that pseudo--rational singularities are rational, while this is open in positive characteristic. We prove a $\bQ_p$--version of this statement in full generality, and a Witt vector version in dimension up to $3$. We also prove a $\bQ_p$--version of Kov\'acs' characterization of rational singularities in characteristic zero \cite{Kovacs_Characteristion_of_rational_singularities}. This is all contained in \autoref{thm_notions_of_singularities}.

Another application is to show some Cohen--Macaulayness/rationality properties of certain singularities. It seems to be widely believed that $F$--rational singularities are rational, but apart from surfaces and \cite{Rulling_Chatsistamatiaou_Higher_direct_image_in_positive_characteristic, Ishii_Yoshida_On_vanishing_of_higher_direct_images_of_the_structure_sheaf}, we are not aware of any theorem that proves some sort of vanishing in this direction. We show the following:

\begin{thm_letter}
	Let $X$ be a variety over a perfect field with $F$--rational singularities. Then $X$ has $\bQ_p$--rational singularities.
	
	If furthermore $\dim(X) = 3$, then $X$ has both Witt--rational singularities and $\bF_p$--rational singularities (i.e. if $\pi \colon Y \to X$ is a resolution, then the sheaves $R^i\pi_*\cO_Y$ with $i > 0$ are nilpotent under the natural action of the Frobenius).
\end{thm_letter}

We find it particularly interesting that using Witt vector methods, we are able to deduce $\bF_p$--rationality; a notion that does not involve Witt vectors whatsoever. It would be interesting to push these ideas and show actual rationality of threefold $F$--rational singularities.

In fact, this theorem also holds to the recently defined quasi--$F$--rational singularities (see \cite{Quasi_F_splittings_III}). It would also be interesting to apply these techniques to pseudo--rational singularities. \\

Let us move to our last application. It is known that klt threefold singularities are rational in characteristic $p > 5$ (\cite{Bernasconi_Kollar_vanishing_theorems_for_threefolds_in_char_>5}) and both $W\cO$--rational and $\bF_p$--rational in all characteristic (\cite{Hacon_Witaszek_On_the_relative_MMP_for_threefolds_in_low_char, Baudin_Bernasconi_Kawakami_Frobenius_GR_fails}). However, there was no Cohen--Macaulayness statement known in characteristics $p \leq 5$. For example, they are \emph{not} Cohen--Macaulay in general, not even up to nilpotence under the Cartier operator (\cite{Totaro_The_failure_of_KV_and_terminal_singularities_that_are_not_CM, Totaro_Terminal_3folds_that_are_not_CM, Baudin_Bernasconi_Kawakami_Frobenius_GR_fails}). Using Witt--GR vanishing, we are finally able to obtain a Cohen--Macaulayness statement that holds in every characteristic.

\begin{thm_letter}
	Let $X$ be a threefold with klt singularities over a perfect field. Then $X$ is Witt--Cohen--Macaulay (see \autoref{def_Witt-CM_and_so_on}).
\end{thm_letter}

We hope that, analogously to the characteristic zero picture, we will be able to show that klt singularities are Witt--rational and Witt--Cohen--Macaulay in arbitrary dimension (perhaps up to the existence of one log resolution). Taking cones and using the trace formula for crystalline/rigid cohomology, this should show that a klt Fano variety over a finite field always has a rational point. The smooth case was solved in \cite{Esnault_Varieties_over_a_finite_field_with_trivial_Chow_group_of_0_cycles_have_a_rational_point}. It would also be interesting to see whether these vanishing theorems can be used to study deformations of Witt--rational and $\bQ_p$--rational singularities. \\

In \cite{Baudin_On_the_Euler_characteristic_of_weakly_ordinary_varieties_of_maximal_albanese_dimension}, we will apply $\bQ_p$--GR vanishing to study the Euler characteristic of ordinary varieties of maximal Albanese dimension in positive characteristic. 

\subsection{Ideas from the proof}

The key theory to tackle our Witt version of Kawamata--Viehweg vanishing (in short, Witt--KVV) is the use of crystalline cohomology (or more precisely rigid cohomology in our case). The idea is as follows: the cohomology of $WI_D$ on $Y$ is intimately related to the rigid cohomology of $V = Y \setminus D$ (see \cite{Berthelot_Bloch_Esnault_On_Witt_vector_cohomology_for_singular_varieties}). Furthermore by the assumption on $D$, there is a proper birational morphism $V \to U$ to some affine variety $U$. The steps are then the following:

\begin{enumerate}
	\item\label{step_Artin} Prove a rigid cohomology version of Artin vanishing for affine varieties, namely that an affine variety $U$ satisfies that $H^i_{\rig}(U/K) = 0$ for $i > \dim(U)$.
	\item\label{step_transfer_Artin} Given a proper birational morphism $\pi \colon V \to U$ with $U$ affine, transfer the Artin vanishing statement above via $\pi$ to deduce some vanishing on $V$.
	\item\label{step_Hodge} Rephrase this last vanishing into Witt--KVV (this is where smoothness is used).
\end{enumerate}

Step \autoref{step_transfer_Artin} is a consequence of descent results of rigid cohomology (\cite{Tsuzuki_Cohomological_descent_of_rigid_cohomology_for_proper_coverings}), while step \autoref{step_Hodge} is a consequence of the theory developed in \cite{Berthelot_Bloch_Esnault_On_Witt_vector_cohomology_for_singular_varieties} and Poincaré duality for rigid cohomology.

The main tool to complete step \autoref{step_Artin} is the theory of arithmetic $\cD$--modules (see \autoref{section_review_D_modules}). Very briefly, this is a theory of coefficients for rigid cohomology, meaning that there are objects on varieties, for which we can take pushforwards, pullbacks and so on (i.e. a six--functor formalism), and there is a ``constant object'' whose pushforward to the base field gives rigid cohomology. This is analogous to how regular holonomic $\cD$--modules are a theory of coefficients for de Rham cohomology in characteristic zero, or how constructible $l$--adic sheaves are a theory of coefficients for $l$--adic cohomology. For example, the proof of Artin vanishing in étale cohomology from \cite[Theorem 7.2]{Milne_Etale_Cohomology} relies heavily on such a formalism. 

It is also in the setup of arithmetic $\cD$--modules that a version of Artin vanishing was already known, solving the first step (note that if $V$ was smooth, then Artin vanishing is much more classical, since in this case rigid cohomology agrees with Monsky-Washnitzer cohomology). With all these tools, we are able to deduce Witt--KVV. \\

Let us finish with the proof of Witt--GR vanishing. As apparent in the phrasing of \autoref{Witt-GR_intro} and the remark thereafter, there are in fact two distinct vanishings in play. For reasons explained in \autoref{explanation_name}, we decide to name the vanishing of $R^i\pi_*W\omega_{Y, \bQ}$ as $\bQ_p$--GR vanishing, while we reserve the name Witt--GR vanishing for the statement that for some $e > 0$, we have $p^e \cdot R^i\pi_*W_n\omega_Y = 0$ for all $i > 0$ and $n \geq 1$. \\

With this in mind, let us discuss how we can deduce Witt--GR (or $\bQ_p$--GR) from Witt--KVV. In characteristic zero, the analogue is rather straightforward (\cite[Corollary 2.68]{Kollar_Mori_Birational_geometry_of_algebraic_varieties}): take an ample divisor $A$ such that all the sheaves $R^i\pi_*\omega_Y \otimes \cO_X(A)$ are globally generated and have no higher cohomology. Then by the degeneration of the Leray spectral sequence, $H^0(X, R^i\pi_*\omega_Y \otimes \cO_X(A)) = H^i(Y, \omega_Y \otimes \cO_Y(\pi^*A) )$. Since the latter group vanishes for $i > 0$ by KVV, we deduce that $R^i\pi_*\omega_Y = 0$ for all $i > 0$.

In our setup, the proof is significantly more complicated. To mimic the proof, we would want to find an divisor $A$ such that the ``twist'' of all the sheaves $R^i\pi_*W_n\omega_Y$ by 
$A$ have vanishing higher cohomology groups and are globally generated. The issue is that there are infinitely many such sheaves to deal with here, so it is very unclear whether such a divisor exists. This is the reason why we imposed additional assumptions in our theorem for Witt--GR vanishing.

Fortunately, the situation can be handled much better in the context of $\bQ_p$--GR vanishing, and such a divisor $A$ can be found. The point is that $\bQ_p$--GR vanishing is equivalent to the following statement: for some $e > 0$, the Cartier operator on $p^e \cdot R^i\pi_*W_n\omega_Y$ is nilpotent for all $i > 0$ and $n \geq 1$. This means that $\bQ_p$--GR is intimately related to the theory of \emph{Witt--Cartier--crystals} (see \autoref{section_Witt_Cartier_crystals}). A great feature of this theory is that, unlike mere coherent modules, every object in this category has \emph{finite length} (see \cite{Blickle_Bockle_Cartier_modules_finiteness_results, Baudin_Duality_between_Witt_Cartier_crystals_and_perverse_sheaves}). In particular, they are built from their simple subquotients, and there are only finitely many of them (this is very wrong in the category of coherent modules). Since all the simple subquotients of the \emph{crystal} $R^i\pi_*W_n\omega_Y$ (with $i$ fixed) are also simple subquotients of $R^i\pi_*\omega_Y$ (\autoref{everyone_has_the_same_simple_subquot}), there are only finitely many sheaves that we truly have to work with. This is the reason we are able to obtain $\bQ_p$--GR vanishing in full generality. 

\subsection{Acknowledgments}

I am very grateful to Kay Rülling, Hiromu Tanaka and Fabio Bernasconi for reading through a previous draft of this article, and to Bhargav Bhatt for simplifying my former proof of \autoref{iso_in_top_slopes_for_proper_birational_morphisms}. I would also like to thank Hélène Esnault, Pascal Fong, Tatsuro Kawakami, Christopher Lazda, Léo Navarro Chafloque, Zsolt Patakfalvi, Karl Schwede, Shunsuke Takagi and Jakub Witaszek for fruitful conversations related to the content of this article. I would also like to thank the anonymous referee for their careful comments.

Financial support was provided by grant $\#$200020B/192035 from the Swiss National Science Foundation (FNS/SNF), and by grant $\#$804334 from the European Research Council (ERC).

\subsection{Notations}\label{sec_notations}
\begin{itemize}
	\item Throughout, we fix a prime number $p > 0$, and a perfect field $k$ of characteristic $p$.
	
	\item Given a scheme $X$ and $n \geq 1$, $W_n\cO_X$ denotes the ($p$--typical) $n$--truncated Witt vectors, together with its usual operators Frobenius, Verschiebung and restriction operators (respectively denoted $F$, $V$ and $R$). We also write $W\cO_X = \lim W_n\cO_X$, and set $K \coloneqq \Frac(W(k))$.
	
	Furthermore, the pair consisting of the underlying topological space of $X$ and the sheaf of rings $W_n\cO_X$ gives a scheme, which we denote by $W_nX$. If $f \colon X \to Y$ is a morphism of schemes in positive characteristic, then the induced morphism $W_nX \to W_nY$ will still be denoted $f$.
	
	\item A variety is a connected, separated scheme of finite type over $k$.
	
	\item A generically finite (resp. birational) morphism $f \colon Y \to X$ of varieties is a morphism such that for some dense open $U \inc X$, the induced map $f^{-1}(U) \to U$ is finite (resp. an isomorphism). An alteration is a proper, surjective and generically finite morphism of varieties.
	
	\item Given a scheme $X$ and an open subscheme $U$ (resp. a closed subscheme $Z$) of $X$, we denote by $j_U \colon U \inj X$ the corresponding open immersion (resp. $i_Z \colon Z \inj X$ the corresponding closed immersion).

	\item Given a separated morphism of finite type of Noetherian schemes $f \colon X \to Y$, we denote by $f^!$ the upper--schriek functor defined in \stacksproj{0A9Y}.
	\item Given an abelian category $\cA$, we denote by $D(\cA)$ its derived category, and $D^b(\cA)$ denotes the full subcategory consisting of complexes $\cM^{\bullet} \in D(\cA)$ such that $\cH^i(\cM^{\bullet}) = 0$ for $\abs{i} \gg 0$.
	\item Given an abelian category $\cA$, consider the Serre subcategory of objects $M \in \cA$ such that $n \cdot M = 0$ for some $n \geq 1$. The associated quotient category (see \stacksproj{02MS}) is denoted $\cA_{\bQ}$. The image of $M \in \cA$ by $\cA \to \cA_{\bQ}$ is denoted $M_{\bQ}$. \\
	
	Let us make a remark here. Say we work with $\cA = \Mod(W\cO_X)$, and let $\cM \in \cA$. We want to point that there is an important difference between $\cM_{\bQ}$ and $\cM \otimes \bQ$. Namely, $\cM_{\bQ} = 0$ if and only if $\cM$ is killed by some fixed integer, while $\cM \otimes \bQ = 0$ if and only if $\cM$ is the union of sub--modules killed by some integer (but the integer may vary!). If we worked with objects with reasonable finiteness conditions, working in these two setups would be the same. However, for example if $\pi \colon Y \to X$ is a proper birational morphism, it is not known whether the sheaves $R^i\pi_*W\cO_Y$ satisfy some finiteness conditions, so we prefer working with this quotient category instead of twisting by $\bQ$ (see also the discussion in \cite[Section 3.1]{Patakfalvi_Zdanowicz_Ordinary_varieties_with_trivial_canonical_bundle_are_not_uniruled}).
\end{itemize}

\section{Preliminaries on Witt theory}

Here is a general lemma about Witt schemes that we will use throughout, without further mention.

\begin{lem}
	Let $f \colon X \to Y$ be a morphism of Noetherian, $F$--finite $\bF_p$--schemes.
	\begin{itemize}
		\item For all $n \geq 1$, the scheme $W_nX$ is also Noetherian, and the induced Frobenius $F \colon W_nX \to W_nX$ is finite.
		\item If $f$ is of finite type (resp. separated, proper), then so is the induced map $W_nX \to W_nY$.
	\end{itemize}
\end{lem}
\begin{proof}
	This is all contained in \cite[Appendix A.1]{Langer_Zink_De_Rham_Witt_cohomology_for_a_proper_and_smooth_morphism}.
\end{proof}

\subsection{Witt--Cartier modules and crystals}\label{section_Witt_Cartier_crystals}

Throughout, $X$ denotes a Noetherian and $F$--finite $\bF_p$-scheme.

\begin{defn}
	\begin{itemize}
		\item A \emph{$W_n$--Cartier module} is a pair ($\cM$, $\theta$), where $\cM$ is a quasi-coherent $W_n\cO_X$-module and $\theta \colon F_*\cM \to \cM$ is a morphism. We call $\theta$ the \emph{structural morphism} of $\cM$. 
		
		\item A morphism of $W_n$--Cartier modules $h \colon (\cM_1, \theta_1) \to (\cM_2, \theta_2)$ is a morphism of underlying $\cO_X$-modules $h \colon \cM_1 \to \cM_2$ making the square
		
		\[ \begin{tikzcd}
			F_*\cM_1 \arrow[rr, "F_*h"] \arrow[d, "\theta_1"'] && F_*\cM_2 \arrow[d, "\theta_2"] \\
			\cM_1 \arrow[rr, "h"']                                     && \cM_2                               
		\end{tikzcd} \]
		commute. The category of $W_n$--Cartier modules is denoted $\QCoh_{W_nX}^C$. The full subcategory of coherent $W_n$--Cartier modules (i.e. which are coherent as $W_n\cO_X$--modules) is denoted $\Coh_{W_nX}^C$.
	\end{itemize}	
\end{defn}

When $n = 1$, we will simply call them \emph{Cartier modules}. \\

\begin{example}
	Arguably the most important example is the canonical $W_n$--dualizing sheaf $W_n\omega_X$, together with the Cartier operator, see \autoref{section_Witt_dc}.
\end{example}

\begin{rem}\label{pushfoward}
	\begin{itemize}
		\item We can take the (higher) pushforward of a $W_n$--Cartier module: if $f \colon X \to Y$ is a morphism of Noetherian $F$--finite $\bF_p$-schemes and $(\cM, \theta)$ is a $W_n$--Cartier module on $X$, then $(R^if_*\cM, R^if_*\theta)$ defines a $W_n$--Cartier module on $Y$. 
		\item Let $\cM$ be a $W_n$--Cartier module. Since $F(p) = p \in W_n\cO_X$, the submodules $p\cM$ and $\cM[p]$ are $W_n$--Cartier submodules ($\cM[p]$ denotes the submodule of $p$--torsion elements of $\cM$).
		\item The category $\QCoh_{W_nX}^C$ is a Grothendieck category, has enough injectives, and injectives in this category are also injective in $\QCoh_{W_nX}$ and $\Mod(W_n\cO_X)$ (see \cite[Lemma 4.1.3 and Corollary 4.2.5]{Baudin_Duality_between_Witt_Cartier_crystals_and_perverse_sheaves}). In particular, taking derived functors (e.g. derived pushforwards and cohomology) gives the expected results.
	\end{itemize}
\end{rem}

\begin{notation}\label{rem:def_Cartier_mod}
	Let $(\cM, \theta)$ be a $W_n$--Cartier module. We will abuse notations as follows: \[ \theta^e \coloneqq \theta \circ F_*\theta \circ \dots \circ F^{e-1}_*\theta \colon F^{e}_*\cM \to \cM. \]
\end{notation}

\begin{defn}
	\begin{itemize}
		\item A $W_n$--Cartier module $(\cM, \theta)$ is said to be \emph{nilpotent} if its structural morphism is nilpotent, i.e. $\theta^e = 0$ for some $e \geq 1$. 
		\item A $W_n$--Cartier module $(\cM, \theta)$ is said to be \emph{locally nilpotent} if it is a union of nilpotent $W_n$--Cartier submodules. 
		\item A morphism $f \colon \cM \to \cN$ of $W_n$--Cartier modules is a \emph{nil--isomorphism} (resp. \emph{lnil--isomorphism}) if both $\ker(f)$ and $\coker(f)$ are nilpotent (resp. locally nilpotent).
	\end{itemize}
\end{defn}

\begin{defn}
	Let $(\cM, \theta)$ be a $W_n$--Cartier module. Then the \emph{perfection} of $(\cM, \theta)$ is $\cM^{\perf} \coloneqq \lim_e F^e_*\cM$.
\end{defn}

\begin{lem}\label{Gabber_finiteness}
	Let $\cM$ be a coherent $W_n$--Cartier module on $X$.
	\begin{enumerate}
		\item The inverse system $\{F^{e}_*\cM\}_{e \geq 1}$ satisfies the Mittag--Leffler condition. In particular, $\cM^{\perf} = \Rlim F^{e}_*\cM$ and perfection is exact on coherent $W_n$--Cartier modules.
		\item We have that $\cM$ is nilpotent if and only if $\cM^{\perf} = 0$.
		\item For any morphism $f \colon \cM \to \cN$ of coherent $W_n$--Cartier modules, $f$ is a nil--isomorphism if and only if $f^{\perf}$ is an isomorphism.
	\end{enumerate}
\end{lem}

\begin{proof}
	See \cite[Lemma 4.4.2]{Baudin_Duality_between_Witt_Cartier_crystals_and_perverse_sheaves}.
\end{proof}

\begin{rem}
	The above tells us that taking perfection is a good way to get rid of nilpotent phenomena in the coherent case. However, coherence is \emph{crucial} here: the Cartier module $H^0(\bA_k^1, \omega_{\bA^1_k})$ is locally nilpotent over any $F$--finite field $k$, but its structural map is surjective! (hence its perfection is non--zero). 
\end{rem}

\begin{defn}[{\cite[Definition 4.3.5]{Baudin_Duality_between_Witt_Cartier_crystals_and_perverse_sheaves}}]
	We define the category of \emph{$W_n$--Cartier quasi--crystals} to be the quotient of $\QCoh_{W_nX}^C$ by locally nilpotent objects (see \stacksproj{02MS}), and denote it $\QCrys_{W_nX}^C$.
	
	Similarly, we define the category of \emph{$W_n$--Cartier crystals} to be the quotient of $\Coh_{W_nX}^C$ by nilpotent objects (note that being locally nilpotent and nilpotent agree for coherent $W_n$--Cartier modules), and denote it $\Crys_{W_nX}^C$.
\end{defn}

\begin{rem}
	\begin{itemize}
		\item Intuitively, working in the category of crystals allows us to forget about nilpotent phenomena. For example, a coherent $W_n$--Cartier module is the zero object as a $W_n$--Cartier crystal if and only if it is nilpotent. More generally, a morphism $f \colon \cM \to \cN$ of $W_n$--Cartier modules is an isomorphism at the level of crystals if and only if it is a nil--isomorphism. 
		\item If $\cM, \cN \in \QCoh_{W_nX}^C$ are isomorphic in $\QCrys_{W_nX}^C$, we will write $\cM \sim_C \cN$.
	\end{itemize}
\end{rem}

Here is a fundamental finiteness result about $W_n$--Cartier crystals that will be very useful to us.

\begin{thm}[{\cite{Blickle_Bockle_Cartier_modules_finiteness_results}}]\label{crystals_have_finite_length}
	Let $\cM$ be a $W_n$--Cartier crystal. Then $\cM$ is both a Noetherian and Artinian object. Hence, it has a finite filtration whose subquotients are simple, and given any two such filtrations, the graded pieces are isomorphic (up to reordering). In particular, the set of simple subquotients of $\cM$ (up to isomorphism) is finite. 
\end{thm}

\begin{proof}
	The fact that $\cM$ is both Noetherian and Artinian is exactly \cite[Proposition 4.4.8]{Baudin_Duality_between_Witt_Cartier_crystals_and_perverse_sheaves}. The other statements are formal consequences.
\end{proof}

We finish this subsection by a construction which will be throughout useful for us.

\begin{constr}[{\cite[Construction 4.4.6]{Baudin_Duality_between_Witt_Cartier_crystals_and_perverse_sheaves}}]\label{construction_log_D}
	Let $(\cM, \theta)$ be a $W_n$--Cartier module, and let $I \subseteq \cO_X$ be a quasi--coherent ideal. Define a $W_n$--Cartier module structure on $\HHom_{W_n\cO_X}(W_nI, \cM)$ by sending $f \colon W_nI \to \cM$ to the function sending $s \in W_nI$ to $\theta(f(F(s))) \in \cM$.
\end{constr}

\begin{lem}\label{pushforward_from_open}
	Let $j \colon U \subseteq X$ be an open immersion, let $\cM$ be a $W_n$--Cartier module, and let $I$ be a quasi--coherent ideal cutting out the complement (with any closed subscheme structure). Then \[ \HHom_{W_n\cO_X}(W_nI, \cM) \sim_C j_*\cM|_U. \]
\end{lem}
\begin{proof}
	See \cite[Lemma 4.4.7.(a)]{Baudin_Duality_between_Witt_Cartier_crystals_and_perverse_sheaves}.
\end{proof}

\subsection{Witt dualizing complexes}\label{section_Witt_dc}

Throughout, $X$ denotes a fixed variety with structural morphism $g \colon X \to \Spec k$. Here is our main object of study:

\begin{defn}
	The \emph{canonical $W_n$-dualizing complex} is the complex $W_n\omega_X^{\bullet} \coloneqq g^!\cO_{W_n(k)}$. By \stacksproj{0AA3}, this is a dualizing complex on $W_nX$, and the associated duality functor is written $\bD \coloneqq \cR\HHom_{W_n\cO_X}(-, W_n\omega_X^{\bullet})$.
\end{defn}

\begin{rem}\label{link_with_usual_de_Rham_Witt_cpx_theory}
	Assume that $X$ is smooth. Then this complex agrees with the last piece of the de Rham--Witt complex of Illusie (see \cite{Illusie_Complexe_de_de_Rham_Witt_et_cohomologie_cristalline}) by \cite[Section I]{Ekedahl_Duality_Hodge_Witt}) up to a shift. These objects carry natural maps $R \colon W_{n + 1}\omega_X \to W_n\omega_X$, $F \colon W_{n + 1}\omega_X \to F_*W_n\omega_X$, $V \colon F_*W_n\omega_X \to W_{n + 1}\omega_X$ and $\underline{p} \colon W_n\omega_X \inj W_{n + 1}\omega_X$ (see \cite{Illusie_Complexe_de_de_Rham_Witt_et_cohomologie_cristalline}). As explained in \cite[Step 1 of the proof of Proposition 9.11]{Quasi_F_splittings_III} (see also \cite[Lemma 2.2.4]{Ekedahl_Duality_Hodge_Witt}), applying $\bD$ to these maps respectively gives $\underline{p} \colon W_n\cO_X \to W_{n + 1}\cO_X$ (the unique morphism such that $p = \underline{p} \circ R$), $V$, $F$ (or more precisely $F \circ R$ under our notations) and $R$. In the following, we will treat the non--necessarily smooth case.
\end{rem}

Since $k$ is perfect, the Frobenius map on $W_n(k)$ is an isomorphism. In particular, $F^! = F^{-1}_*$ on $\Spec W_n(k)$, so there is a natural isomorphism \[ \cO_{W_n(k)} \to F^!\cO_{W_n(k)}. \] We can then apply $g^!$ to obtain a natural isomorphism \[ W_n\omega_X^{\bullet} \to F^!W_n\omega_X^{\bullet}. \] Since $F$ is finite on $W_nX$, we obtain by adjunction a morphism \[ C \colon F_*W_n\omega_X^{\bullet} \to W_n\omega_X^{\bullet}, \] which we call the \emph{Cartier operator} on $W_n\omega_X^{\bullet}$. Note that by construction, \[ \bD(C) \colon W_n\cO_X \to \cR\HHom(F_*W_n\omega_X^{\bullet}, W_n\omega_X^{\bullet}) \cong F_*\cR\HHom(W_n\omega_X^{\bullet}, F^!W_n\omega_X^{\bullet}) \cong F_*W_n\cO_X\] is simply the Frobenius map. Now, consider the composition \[ W_n\cO_X \xto{F} F_*W_n\cO_X \xto{V} W_{n + 1}\cO_X. \] Applying $\bD$ and standard properties of Grothendieck duality, we obtain a morphism \[ R \colon W_{n + 1}\omega_X^{\bullet} \xto{\bD(V)} F_*W_n\omega_X^{\bullet} \xto{C} W_n\omega_X^{\bullet}. \]
Note that the diagram \[ \begin{tikzcd}
	F_*W_{n + 1}\omega_X^{\bullet} \arrow[rr, "C"] 	\arrow[d, "F_*R"] &  & W_{n + 1}\omega_X^{\bullet} \arrow[d, "R"] \\
	F_*W_n\omega_X^{\bullet} \arrow[rr, "C"]                 &  & W_n\omega_X^{\bullet}                
\end{tikzcd} \]
commute, because its Grothendieck dual 
	\[ \begin{tikzcd}
	F_*W_{n + 1}\cO_X               &  & W_{n + 1}\cO_X \arrow[ll, "F"']          \\
	F^2_*W_n\cO_X \arrow[u, "F_*V"] &  & F_*W_n\cO_X \arrow[u, "V"']              \\
	F_*W_n\cO_X \arrow[u, "F_*F"]   &  & W_n\cO_X \arrow[ll, "F"] \arrow[u, "F"']
	\end{tikzcd} \] does. In particular, if we denote 
\[ (W_n\omega_X^{\bullet})^{\perf} \coloneqq \Rlim_e F^e_*W_n\omega_X^{\bullet}, \]
where the derived limit is taken over the Cartier operators, we obtain natural maps \[ R^{\perf} \colon (W_{n + 1}\omega_X^{\bullet})^{\perf} \to (W_n\omega_X^{\bullet})^{\perf}. \]

\begin{rem}
	Note that by exactness of perfection of coherent $W_n$--Cartier modules (see \autoref{Gabber_finiteness}), there is no potential clash of notations between what we used to denote $(\cdot)^{\perf}$ and what we write here.
\end{rem}

\begin{defn}
	The \emph{canonical Witt dualizing complex} on $X$ is by definition \[  
	W\omega_X^{\bullet} \coloneqq \Rlim_n W_n\omega_X^{\bullet}, \] where the derived limit is taken over the maps $R \colon W_{n + 1}\omega_X^{\bullet} \to W_n\omega_X^{\bullet}$.
\end{defn}

\begin{lem}\label{Witt_omega_coperfect}
	The induced morphism $C \colon F_*W\omega_X^{\bullet} \to W\omega_X^{\bullet}$ is an isomorphism. In particular, there is a natural isomorphism given by the composition \[ W\omega_X^{\bullet} \cong (W\omega_X^{\bullet})^{\perf} = \Rlim_e \Rlim_n F^e_*W_n\omega_X^{\bullet} \cong  \Rlim_n \: (W_n\omega_X^{\bullet})^{\perf}. \]
\end{lem}
\begin{proof}
	We start with the first statement. Consider the commutative diagram 
	\[ \begin{tikzcd}
		W_{n + 2}\omega_X^{\bullet} \arrow[r, "\bD(V)"] \arrow[rr, "R", bend left] & F_*W_{n + 1}\omega_X^{\bullet} \arrow[r, "C"] \arrow[rr, "F_*R"', bend right] & W_{n + 1}\omega_X^{\bullet} \arrow[r, "\bD(V)"] & F_*W_n\omega_X^{\bullet}.
	\end{tikzcd} \]
	Taking derived limit gives \[ \begin{tikzcd}
		W\omega_X^{\bullet} \arrow[r] \arrow[rr, "id", bend left] & F_*W\omega_X^{\bullet} \arrow[r, "C"] \arrow[rr, "id"', bend right] & W\omega_X^{\bullet} \arrow[r] & F_*W\omega_X^{\bullet},
	\end{tikzcd} \] so $C \colon F_*W\omega_X^{\bullet} \to W\omega_X^{\bullet}$ has an inverse (hence it is an isomorphism). We conclude the proof by the following string of natural isomorphisms:
	\begin{align*}
		W\omega_X^{\bullet} & \cong \Rlim_eF^e_*W\omega_X^{\bullet} \\
		& = \Rlim_eF^e_*\Rlim_n W_n\omega_X^{\bullet} \\
		& \cong \Rlim_e \Rlim_n F^e_*W_n\omega_X^{\bullet} \\
		& \cong \Rlim_n \Rlim_e F^e_*W_n\omega_X^{\bullet} \\
		& = \Rlim_n \: (W_n\omega_X^{\bullet})^{\perf}. \qedhere
	\end{align*} 
\end{proof}

\begin{lem}\label{explicit_description_mult_by_p}
	Multiplication by $p^m$ on $W_{n + m}\omega_X^{\bullet}$ is given by the composition \[ W_{n + m}\omega_X^{\bullet} \xto{R^m} W_n\omega_X^{\bullet} \xto{\bD(R^m)} W_{n + m}\omega_X^{\bullet}, \] where we recall that in the second arrow, $R^m \colon W_{n + m}\cO_X \to W_n\cO_X$ denotes the map $(s_1, \dots, s_{n + m}) \mapsto (s_1, \dots, s_n)$.
\end{lem}

\begin{proof}
	This follows immediately by Grothendieck duality, since multiplication by $p^m$ on $W_{n + m}\cO_X$ is given by \[ W_{n + m}\cO_X \xto{R^m} W_{n}\cO_X \xto{F^m} F^m_*W_n\cO_X \xto{V^m} W_{n + m}\cO_X. \qedhere \]
\end{proof}

\begin{lem}\label{cone_of_p_Witt_omega}
	The following is an exact triangle:
	\[ \begin{tikzcd}
		W\omega_X^{\bullet} \arrow[r, "p^m"] & W\omega_X^{\bullet} \arrow[r] & (W_m\omega_X^{\bullet})^{\perf} \arrow[r, "+1"] & {}
	\end{tikzcd} \] where the second map is given by the composition $W\omega_X^{\bullet} \xto{\cong} \Rlim_n \: (W_n\omega_X^{\bullet})^{\perf} \to (W_m\omega_X^{\bullet})^{\perf}$.
\end{lem}

\begin{proof}
	We have a commutative diagram of exact triangles \[ \begin{tikzcd}
		W_{n + 1}\omega_X^{\bullet} \arrow[d, "R"'] \arrow[rr, "\bD(R^m)"] &  & W_{n + 1 + m}\omega_X^{\bullet} \arrow[d, "R"'] \arrow[rr, "\bD(V^{n + 1})"] &  & F^{n + 1}_*W_m\omega_X^{\bullet} \arrow[d, "C"'] \arrow[rr, "+1"] &  & {} \\
		W_n\omega_X^{\bullet} \arrow[rr, "\bD(R^m)"]                       &  & W_{n + m}\omega_X^{\bullet} \arrow[rr, "\bD(V^n)"]                           &  & F^n_*W_m\omega_X^{\bullet} \arrow[rr, "+1"]                       &  & {}
	\end{tikzcd} \] (checking its commutativity can be done by applying $\bD$). Applying $\Rlim$ and \autoref{explicit_description_mult_by_p} concludes the proof.
\end{proof}

\begin{defn}
	\begin{itemize}
		\item The \emph{canonical $W_n$--dualizing sheaf} is $W_n\omega_X \coloneqq \cH^{-\dim(X)}(W_n\omega_X^{\bullet})$, and we write $W_n\omega_X^{\perf} \coloneqq (W_n\omega_X)^{\perf}$.
		\item The \emph{canonical Witt--dualizing sheaf} is $W\omega_X \coloneqq \lim_n W_n\omega_X$. 
	\end{itemize}
\end{defn}

\begin{lem}\label{Witt_omega_p_torsion_free}
	The sheaf $W\omega_X$ is $p$-torsion free.
\end{lem}
\begin{proof}
	This follows from \autoref{cone_of_p_Witt_omega}, since \[ \cH^{-\dim(X)}(W\omega_X^{\bullet}) = \cH^{-\dim(X)}(\Rlim_n W_n\omega_X^{\bullet}) = \lim_n \cH^{-\dim(X)}(W_n\omega_X^{\bullet}) = W\omega_X. \qedhere \] 
\end{proof}

\begin{lem}\label{general_Mittag_Leffler_property}
	Let $\{\cM_n\}_{n \geq 1}$ be an inverse system of sheaves on a Noetherian $F$--finite $\bF_p$--scheme $Y$, where each $\cM_n$ is a coherent $W_n$--Cartier module and each map $\cM_{n + 1} \to \cM_n$ is a morphism of $W_{n + 1}$--Cartier modules. Then the inverse system of sheaves $\{ \cM_n^{\perf} \}_{n \geq 1}$ satisfies the Mittag--Leffler condition.
\end{lem}
\begin{proof}
	Fix $n \geq 1$, and given $m \geq n$, set $\cI_m \coloneqq \im(\cM_m \to \cM_n)$. By \cite[Lemma 4.3.8]{Baudin_Duality_between_Witt_Cartier_crystals_and_perverse_sheaves}, we can see these sub--$W_m$--Cartier modules of $\cM_n$ as sub--$W_n$--Cartier crystals. By Artinianity of $W_n$--Cartier crystals (see \autoref{crystals_have_finite_length}), these images stabilize up to nilpotence. Since perfection is exact and annihilates nilpotence (\autoref{Gabber_finiteness}), we conclude that the system $\{W_n\omega_X^{\perf}\}_n$ satisfies the Mittag--Leffler condition.
\end{proof}

\begin{lem}\label{Witt_omega_Mittag_Leffler}
	The natural maps $W\omega_X \to \Rlim_n W_n\omega_X$ and $W\omega_X \to \Rlim_n W_n\omega_X^{\perf}$ are isomorphisms.
\end{lem}

\begin{proof}
	By \autoref{Witt_omega_coperfect}, it is enough to show that the higher $R^i\lim$ vanish. The fact that $R^i\lim W_n\omega_X^{\perf}$ vanishes for $i > 0$ follows readily from \autoref{general_Mittag_Leffler_property}, so we are left to show that $R^i\lim W_n\omega_X = 0$ for $i > 0$. Applying $\cH^{-\dim(X)}$ to the first diagram of the proof of \autoref{Witt_omega_coperfect} and then applying $\Rlim$ shows as in \emph{loc. cit.} that the induced Cartier operator on $\Rlim_n W_n\omega_X$ is an isomorphism, and hence that the natural map \[ \Rlim_e \Rlim_n F^e_*W_n\omega_X = \Rlim_e F^e_*\Rlim_n W_n\omega_X \to \Rlim_n W_n\omega_X \] is an isomorphism (the inverse limit indexed by $e$ is taken over the Cartier operator). Since $W_n\omega_X^{\perf} = \Rlim_e F^e_*W_n\omega_X$ by \autoref{Gabber_finiteness}, we deduce by switching the $\Rlim$'s that the natural map $\Rlim_n (W_n\omega_X)^{\perf} \to \Rlim_n W_n\omega_X$ is an isomorphism.
\end{proof}

\begin{rem}\label{link_with_usual_full_de_Rham_Witt_cpx_theory}
	Assume that $X$ is smooth. Then the canonical Witt--dualizing complex agrees with the last piece of the (non--truncated) de Rham--Witt complex of Illusie (see \cite{Illusie_Complexe_de_de_Rham_Witt_et_cohomologie_cristalline}) up to a shift. Indeed, by \autoref{link_with_usual_de_Rham_Witt_cpx_theory}, this is true for each canonical $W_n$--dualizing complex, and also our maps $R \colon W_{n + 1}\omega_X^{\bullet} \to W_n\omega_X^{\bullet}$ agree with the restriction maps in \cite{Illusie_Complexe_de_de_Rham_Witt_et_cohomologie_cristalline}. Hence, we conclude by \autoref{Witt_omega_Mittag_Leffler}. 
\end{rem}

\begin{notation}
	Let $D$ be an effective Cartier divisor, with associated ideal sheaf $I_D \inc \cO_X$. Then we will write \[ W_n\omega_{(X, D)} \coloneqq \HHom_{W_n\cO_X}(W_nI_D, W_n\omega_X) \] (see \autoref{construction_log_D}).
\end{notation}
\begin{rem}
	When $X$ is smooth and $D$ has simple normal crossing support, then $W_n\omega_{(X, D)}$ agrees with the sheaf $W_n\Omega_{(X, D)}^{\dim X}$ defined in \cite[6.2.1]{Rulling_Ren_Duality_for_Hodge_Witt_cohomology_with_modulus} (this follows from Definition 8.1 and Theorem 9.3 in \emph{loc. cit.}). In particular, if $D$ is in addition reduced, then we obtain the usual sheaf $W_n\Omega_X^{\dim X}(\log D)$ of top Witt--forms with logarithmic poles along $D$ (see e.g. \cite{Matsuue_On_relative_and_overconvergent_de_Rhamm_Witt_cohomology_for_log_schemes}) by \cite[Remark 5.5.3]{Rulling_Ren_Duality_for_Hodge_Witt_cohomology_with_modulus}.
\end{rem}

\begin{lem}\label{logarithmic_works_well_when_CM}
	Assume that $X$ is Cohen--Macaulay, and let $D$ be an effective Cartier divisor. Then for all $n > 0$, we have that \[ \cR\HHom_{W_n\cO_X}(W_nI_D, W_n\omega_X) = W_n\omega_{(X, D)} \] (see \autoref{construction_log_D}). In particular, for any $n > 0$, we have short exact sequences \[ 0 \to W_n\omega_{(X, D)} \to W_{n + 1}\omega_{(X, D)} \to F^n_*\omega_X(D) \to 0. \]
\end{lem}
\begin{proof}
	First, note that $\omega_X = \omega_X^{\bullet}[-\dim X]$ by Cohen--Macaulayness of $X$. Thanks to the exact triangles \[ \begin{tikzcd}
		W_n\omega_X^{\bullet} \arrow[rr, "\bD(R)"] &  & W_{n + 1}\omega_X^{\bullet} \arrow[rr, "\bD(V^n)"] &  & F^n_*\omega_X^{\bullet} \arrow[rr, "+1"] &  & {}
	\end{tikzcd} \] we obtain that also $W_n\omega_X = W_n\omega_X^{\bullet}[-\dim X]$.
 
	Now, fix $n \geq 1$, and let us show by induction on $1 \leq m \leq n$ that for all $i > 0$, we have that $\EExt^i_{W_n\cO_X}(W_mI_D, W_n\omega_X) = 0$ (this shows the first result by definition). If $m = 1$, then by Grothendieck duality and our observation above, \[ \cR\HHom_{W_n\cO_X}(I_D, W_n\omega_X) \cong \cR\HHom_{\cO_X}(I_D, \omega_X) \expl{\cong}{$I_D = \cO_X(-D)$ is locally free} \omega_X(D), \] so the result holds. For $2 \leq m \leq n$, consider the exact sequence 
	\begin{equation}\label{hi_ppl_reading_my_tex____disgusting_isnt_it}
		\begin{tikzcd}
			0 \arrow[rr] &  & F^{m - 1}_*I_D \arrow[rr, "V^m"] &  & W_{m}I_D \arrow[rr, "R"] &  & W_{m - 1}I_D \arrow[rr] &  & 0.
		\end{tikzcd} 
	\end{equation} Applying $\cR\HHom_{W_n\cO_X}(-, W_n\omega_X)$ gives an exact triangle \[ \begin{tikzcd}
		{\cR\HHom(W_{m - 1}I_D, W_n\omega_X)} \arrow[r] &  {\cR\HHom(W_{m}I_D, W_n\omega_X)} \arrow[r] &  {\cR\HHom(F^{m - 1}_*I_D, W_n\omega_X)} \arrow[r, "+1"] &  {}
	\end{tikzcd} \] so given that the left term lives in degree zero by induction, it is enough to show that the right term does too. Since 
	\begin{align*}
		\cR\HHom_{W_n\cO_X}(F^{m - 1}_*I_D, W_n\omega_X) & \cong F^{m - 1}_*\cR\HHom_{W_n\cO_X}(I_D, F^{m  - 1, !}W_n\omega_X) \\ & \expl{\cong}{$F^!W_n\omega_X^{\bullet} 	\cong W_n\omega_X^{\bullet}$ and $W_n\omega_X^{\bullet} = W_n\omega_X[\dim X]$} F^{m - 1}_*\cR\HHom_{W_n\cO_X}(I_D, W_n\omega_X), 
	\end{align*} the proof of the induction step is finished by the case $m = 1$. 
	
	The statement after ``In particular'' now follows by applying $\cR\HHom_{W_{n + 1}\cO_X}(-, W_{n + 1}\omega_X)$ to the short exact sequence \autoref{hi_ppl_reading_my_tex____disgusting_isnt_it} with $m = n + 1$.
\end{proof}

\section{Witt--Kawamata--Viehweg vanishing}

\subsection{Review of arithmetic $\cD$-modules}\label{section_review_D_modules}

We will use, as much as possible, the notations from \cite{Lazda_Rigidification_of_arith_D_mods_and_overconv_RH_corr}. In the ``analytic'' side of the proof, our main tool will be rigid cohomology. This theory was introduced  by Berthelot in \cite{Berthelot_Geometrie_rigide_et_cohomologie_des_varieties_algebriques_de_caracteristique_p}, it takes values in fields of characteristic zero (hence it is not integral), and mimics de Rham cohomology in characteristic zero. Unlike crystalline cohomology, it has the advantage that it can be applied to any variety (and gives a reasonable theory), not only smooth proper ones. 

In fact, rigid cohomology can be taken with coefficients in any overconvergent $F$--isocrystal (they can be thought of as vector bundles on a characteristic zero lift, together with  an integrable connection and a Frobenius action, and with stronger convergence properties at the boundary for non-proper varieties). These objects are analogous to lisse sheaves in $l$-adic cohomology, or simply vector bundles with an integrable connection in characteristic zero. In particular, they cannot be expected to be stable under a six--functor formalism (non--smooth pushforwards are really an issue).

Constructing a full theory of coefficients for rigid cohomology has been a very rich subject in the last 30 years (\cite{Berthelot_D_modules_arithmetiques_I, Berthelot_D_modules_arithmetiques_II, Caro_D_modules_arithmetiques_surholonomes, Caro_Stabilite_de_l_holonomie_sur_les_vars_quasi_projs, Caro_Tsuzuki_Overholonomicity_of_overconvergent_F_isocrystals_over_smooth_varieties, Abe_Langlands_correspondence_for_isocrystals, Abe_Caro_Theory_of_weights_in_p_adic_cohomology,  Abe_Around_the_nearby_cycle_functor_for_arithmetic_D-modules, Lazda_Rigidification_of_arith_D_mods_and_overconv_RH_corr}). After this long work, we finally have objects, containing overconvergent $F$--isocrystals, which are stable under the six operations (namely pushforwards, pullbacks, tensor products, duality, exceptional direct images and exceptional inverse images), and which recover rigid cohomology: they are called \emph{overholonomic $\cD^{\dagger}$--modules}. There are also two natural t-structures on the derived category of these objects, whose analogues in characteristic zero correspond to the standard and perverse t-structures on constructible sheaves via the Riemann--Hilbert correspondence.

Let us introduce all the notations we will need:

\begin{notation}
	Let $X$ be a variety over $k$.
	\begin{itemize}
		\item It is said to be \emph{strongly realizable} if $X$ admits a locally closed immersion into a proper and smooth formal scheme over $\Spf W(k)$. In particular, quasi-projective varieties are strongly realizable (we will only ever use this theory in this context).
	\end{itemize}

From now on, assume that $X$ is strongly realizable. 
	\begin{itemize}
		\item The (bounded) derived category of overholonomic $\cD^{\dagger}$-modules on $X$ is denoted $D^b_{\hol}(X)$.
		\item The six--functors of this theory (see \cite[Section 1.1]{Abe_Langlands_correspondence_for_isocrystals}) on $D^b_{\hol}(X)$ will be denoted as follows: pushforwards by $f_+$, pullbacks by $f^+$, tensor product by $\otimes$, duality by $\bD$, exceptional direct image by $f_!$ and exceptional inverse image by $f^!$. We warn the reader that all these functors are already derived.
		
		\item We shall write \[ \bunit_X \coloneqq g^+(K), \] where $K$ denotes the constant arithmetic $\cD$-module on $k$ and $g \colon X \to \Spec k$ is the structural map.
		
		It is typically denoted $K_X$ in references, but this clashes with the usual notation from birational geometry, so we decided to give it another name here (the symbol $\bunit$ comes from the fact that it is the unit with respect to $\otimes$).
		
		\item The \emph{standard t-structure} on $D^b_{\hol}(X)$, denoted ($D^{\leq 0}(X)$, $D^{\geq 0}(X)$) is the one defined in \cite[Definition 1.2.1]{Abe_Caro_Theory_of_weights_in_p_adic_cohomology} (see 1.3.14 in \emph{loc. cit.}). The truncation functors will be denoted $\cH^i$. This t-structure should be thought of as an analogue of the perverse t-structure on constructible $l$-adic sheaves. 
		
		\item The \emph{constructible t-structure} on $D^b_{\hol}(X)$, denoted (${}^cD^{\leq 0}(X), {}^cD^{\geq 0}(X)$) is the one defined in \cite[1.3]{Abe_Langlands_correspondence_for_isocrystals}. %		Although it is not exactly defined on $F$-$D^b_{\ovhol}(X/K)$, it follows from the proof of the proof of Proposition 1.3.3 in \emph{loc.cit.} that truncations commute with Frobenius pullbacks, hence restricting to a t-structure on this subcategory too. 
		This t-structure should be thought of as an analogue of the standard t-structure on constructible $l$-adic sheaves.
	\end{itemize} 
\end{notation}

\begin{lem}\label{comparison_t_structures}
	Let $X$ be a strongly realizable variety. Then ${}^cD^{\leq 0}(X) \inc D^{\leq \dim X}(X)$.
\end{lem}
\begin{proof}
	This is automatic from the definitions. Indeed, the requirement for some $\cM^{\bullet}$ to be in ${}^cD^{\leq 0}(X)$ is that for all closed immersion $i \colon Z \inj X$, we have $i^+\cM^{\bullet} \in D^{\leq \dim Z}(Z)$. Applying this to $Z = X$ gives the result.
\end{proof}

In the following, $D^b_{\mathrm{fg}}(K)$ denotes the derived category of finitely generated $K$--vector spaces.

\begin{lem}\label{comparison_for_k}
	The ``forgetful functor'' from overholonomic $\cD^{\dagger}$-modules on $k$ to finitely generated $K$--vector spaces is an equivalence of categories. Furthermore, the constructible and standard t-structures agree on $D^b_{\hol}(k)$, and they correspond to the standard t-structure on $D^b_{\mathrm{fg}}(K)$.
\end{lem}
\begin{proof}
	This follows immediately from the definitions.
\end{proof}

As already mentioned, the theory of arithmetic $\cD$-modules is supposed to enlarge the theory of rigid cohomology. This is made precise by the following important theorem:

\begin{thm}[\cite{Lazda_Rigidification_of_arith_D_mods_and_overconv_RH_corr}]\label{Lazda_comparison}
	Let $X$ be a strongly realizable variety, with structural map $g \colon X \to \Spec k$. Then there is an isomorphism \[ g_+\bunit_X \cong \RGamma_{\rig}(X/K). \]
\end{thm}
\begin{proof}
	By \cite[Last Corollary in Introduction]{Lazda_Rigidification_of_arith_D_mods_and_overconv_RH_corr}, we know that if $\cF \coloneqq g^*K$ denotes the constant overconvergent $F$--isocrystal on $X$, then \[ \RGamma_{\rig}(X/K) \coloneqq \RGamma_{\rig}(X, \cF) \cong g_+\rho_X(\cF) \]
	(the functor $\rho$ from overconvergent ($F$-)isocrystals to arithmetic $\cD$-modules is constructed in \cite[Theorem in Section 3.8]{Abe_Around_the_nearby_cycle_functor_for_arithmetic_D-modules}). By \emph{loc. cit.}, we know that $\rho$ commutes with pullbacks, in the sense that \[ \rho_X \circ g^* = g^+ \circ \rho_k. \]
	In our case, this gives the isomorphism \[ \rho_X(\cF) \cong \bunit_X, \] concluding the proof. 
\end{proof}

\subsection{Proof of Witt-Kawamata-Viehweg vanishing}

We start by showing the analogue of Artin vanishing. The author is not aware of a proof in full generality that does not use arithmetic $\cD$--modules.

\begin{thm}\label{vanishing for affines}
	Let $X$ be an affine variety. Then for all $i > \dim(X)$, we have \[ H^i_{\rig}(X/K) = 0. \]
\end{thm}
\begin{proof}
	Let $g \colon X \to \Spec k$ denote the structural map. By \autoref{Lazda_comparison} and \autoref{comparison_for_k}, our goal is to show that for all $i > \dim X$, \[ \cH^i(g_+\bunit_X) = 0. \]
	In other words, we want to show that $g_+\bunit_X \in D^{\leq \dim X}(k)$. Since $g^+$ is $t$-exact for the constructible $t$-structure (\cite[Lemma 1.3.4]{Abe_Langlands_correspondence_for_isocrystals}), we know in particular that $\bunit_X \in D^{\leq \dim X}(X)$ by \autoref{comparison_t_structures}. Since $g$ is an affine morphism, it is right t-exact for the standard $t$--structure by \cite[Proposition 1.3.13.(a)]{Abe_Caro_Theory_of_weights_in_p_adic_cohomology} (which fundamentally comes from the results of \cite{Noot_Huyghe_D_affinite_de_l_espace_projectif}), so $g_+\bunit_X \in D^{\leq \dim X}(k)$.
\end{proof}

Now, let us move to Witt--Kawamata--Viehweg vanishing. We will use the Dieudonné--Manin classification and the theory of slopes of $F$--isocrystals, see e.g. \cite[Beginning of the proof of Theorem 1.1]{Berthelot_Bloch_Esnault_On_Witt_vector_cohomology_for_singular_varieties}.

\begin{notation}
	Given an interval $I \subseteq \bR$ and an $F$--isocrystal $M$ on $K$, we denote by $M^I$ the maximal sub--$F$--isocrystal of $M$ with slopes in $I$. 
\end{notation}

\begin{lem}\label{slopes_of_rigid_cohomology}
	Let $X$ be a separated scheme of finite type over $k$. Then the slopes of its rigid cohomology groups lie in $[0, \dim X]$.
\end{lem}
\begin{proof}
	When $X$ is smooth, this is the special case when $Z = \emptyset$ in \cite[Theorem 16.11]{Nakkajima_Weight_filtration_and_slope_filtration}. In general, this follows by the same hypercovering argument as in the proof of \cite[Theorem 5.1.1]{Tsuzuki_Cohomological_descent_of_rigid_cohomology_for_proper_coverings}.
\end{proof}

\begin{lem}\label{iso_in_top_slopes_for_proper_birational_morphisms}
	Let $\pi \colon Y \to X$ be a proper birational morphism of varieties of dimension $d$. Then for all $i \geq 0$, the natural map \[ H^i_{\rig}(X/K)^{]d - 1, d]} \to H^i_{\rig}(Y/K)^{]d - 1, d]}  \] is an isomorphism.
\end{lem}
\begin{rem}
	The proof below was kindly suggested to us by Bhargav Bhatt. Our original proof used the six--functor formalism on arithmetic $\cD$--modules, thereby making it rely on significantly heavier machinery than that below. In particular, the theory of arithmetic $\cD$--modules is only used to obtain the analogue of Artin vanishing.
\end{rem}
\begin{proof}
	Let $Z \inc Y$ be a closed subscheme of $Y$ over which $\pi$ is not an isomorphism (by assumption, $\dim(Z) < d)$. Since $X \mapsto \RGamma_{\rig}(X/K)$ satisfies descent for proper surjective morphisms by \cite{Tsuzuki_Cohomological_descent_of_rigid_cohomology_for_proper_coverings} (i.e. rigid cohomology defines an $h$--sheaf) it follows from \cite[Theorem 2.9]{Bhatt_Scholze_Projectivity_of_the_Witt_vector_Grassmannian} (see Proposition 2.8 in \emph{loc. cit.} for a non--derived version) that we have a homotopy pullback square
	\[ \begin{tikzcd}
		\RGamma_{\rig}(X/K) \arrow[d] \arrow[r] & \RGamma_{\rig}(Y/K) \arrow[d] \\
		\RGamma_{\rig}(Z/K) \arrow[r]                     & \RGamma_{\rig}(\pi^{-1}(Z)/K).         
	\end{tikzcd} \] 
	Equivalently, this means that we have a ``Mayer--Vietoris'' exact triangle \[ \begin{tikzcd}
				\RGamma_{\rig}(X/K) \arrow[r]  & \RGamma_{\rig}(Y/K) 	\oplus \RGamma_{\rig}(Z/K) \arrow[r] &  \RGamma_{\rig}(\pi^{-1}(Z)/K) \arrow[r, "+1"] &  {}
			\end{tikzcd} \] 
	Since $\pi^{-1}(Z)$ and $Z$ have no slopes $> d - 1$ in their rigid cohomology by \autoref{slopes_of_rigid_cohomology}, we deduce that $\pi^* \colon \RGamma_{\rig}(X/K)^{]d - 1, d]} \to \RGamma_{\rig}(Y/K)^{]d - 1, d]}$ is an isomorphism.
\end{proof}

For the definition and properties of compactly supported Witt vector cohomology, see \cite{Berthelot_Bloch_Esnault_On_Witt_vector_cohomology_for_singular_varieties}.

\begin{cor}\label{vanishing_rigid_cohomology_top_slopes}
	Let $Y$ be a variety of dimension $d$ admitting a proper birational morphism to an affine variety. Then for all $i > d$, \[ H^i_{\rig}(Y/K)^{]d - 1, d]} = 0. \] If $Y$ is in addition smooth, then for all $i < d$, \[ H^i_{\rig, c}(Y/K)^{[0, 1[} = H^i_{c}(Y, W\cO_{Y, K}) = 0. \]
\end{cor}
\begin{proof}
	The first statement follows immediately from \autoref{vanishing for affines} and \autoref{iso_in_top_slopes_for_proper_birational_morphisms}. Now, assume that $Y$ is smooth. Then by Poincaré duality (see e.g. \cite[Corollary 8.3.14]{Le_Stum_Rigid_cohomology} for the definition of the pairing and its compatibility with Frobenius actions, and \cite{Berthelot_Dualite_Poincare_et_Kunneth_en_cohom_rigide} for a proof that it is an isomorphism), we deduce that for all $i < d$, \[  H^i_{\rig, c}(Y/K)^{[0, 1[} = 0. \]
	We then conclude the proof by \cite[Theorem 1.1]{Berthelot_Bloch_Esnault_On_Witt_vector_cohomology_for_singular_varieties}.
\end{proof}

We thank Shunsuke Takagi for suggesting us the following generalization of \autoref{vanishing_rigid_cohomology_top_slopes}.

\begin{cor}\label{Takagi_generalisation}
	Let $Y$ be a smooth quasi--projective variety admitting a proper surjective morphism $\pi \colon Y \to X$ to an affine variety $X$. Then for all $i < \dim(X)$, we have that \[ H^i_{\rig, c}(Y/K)^{[0, 1[} = H^i_c(Y, W\cO_{Y, K}) = 0. \]
\end{cor}
\begin{proof}
	We will prove the result by induction on $\dim(Y) - \dim(X) \geq 0$. If $\dim(Y) - \dim(X) = 0$, then given that we may replace $\pi \colon Y \to X$ by its Stein factorization, the statement is a particular case of \autoref{vanishing_rigid_cohomology_top_slopes}. If $\dim(Y) - \dim(X) > 0$, take a general hyperplane section $H \inc Y$, and set $V \coloneqq Y \setminus H$. By induction, we know that $H^i_{\rig, c}(H/K)^{[0, 1[} = 0$ for all $i < \dim(X)$. By \cite[3.1.(iii)]{Berthelot_Geometrie_rigide_et_cohomologie_des_varieties_algebriques_de_caracteristique_p}, there is an exact triangle \[ \begin{tikzcd}
		{\RGamma_{\rig, c}(V/K)} \arrow[rr] &  & {\RGamma_{\rig, c}(Y/K)} \arrow[rr] &  & {\RGamma_{\rig, c}(H/K)} \arrow[rr, "+1"] &  & {}
	\end{tikzcd} \] so given that $V$ is affine, we conclude by taking slopes and applying \autoref{vanishing_rigid_cohomology_top_slopes} to $V$ that $H^i_{\rig, c}(Y/K)^{[0, 1[} = 0$ for all $i < \dim(X)$.
\end{proof}

Before moving further, let us prove a finiteness result, well--known to experts.

\begin{lem}\label{finiteness_p_torsion}
	Let $Y$ be a proper variety, let $I \inc \cO_Y$ be a quasi--coherent ideal, let $i \geq 0$, and let $M$ denote the $W(k)$--submodule of $H^i(Y, WI)$ of elements which are $p$--power torsion. Then $p^e \cdot M = 0$ for some $e > 0$.
\end{lem}
\begin{proof}
	We include a proof for the sake of completeness (we will mimic the argument in \cite[Proposition 2.10]{Berthelot_Bloch_Esnault_On_Witt_vector_cohomology_for_singular_varieties}). Consider the ring $R \coloneqq W(k)\llbracket V \rrbracket$, subject to the non--commutativity rule $aV = VF(a)$ for all $a \in W(k)$. Since $H^i(Y, WI)$ is separated and complete for the $V$-adic topology, it is naturally endowed with a left $R$--module structure. Note that $H^i(Y, WI)$ is finitely generated as a left $R$--module, since $H^i(Y, WI)/V \inc H^i(Y, I)$ is a finite--dimensional $k$-vector space by properness of $Y$. By (left--)Noetherianity of $R$ (\cite[Lemma 2.3]{Bloch_Algebraic_K_theory_and_crys_cohomology}), the $R$--submodule $M$ is also finitely generated, say with generators $m_1, \dots, m_r \in M$. If $e > 0$ satifies that $p^em_i = 0$ for all $1 \leq i \leq r$, then surely $p^eM = 0$.
\end{proof}

Given a $\bQ$--Cartier divisor $D$ on a normal proper variety $Y$, we denote by $\kappa(Y, D) \in \{-\infty, 0, \dots, \dim(Y)\}$ its \emph{Iitaka dimension} (i.e. the dimension of the image of the rational map induced by the linear system $|mD|$ for $m$ sufficiently big and divisible). For example $D$ is big if and only if $\kappa(Y, D) = \dim(Y)$. We also say that $D$ is \emph{semi--ample} if for some $m > 0$, the linear system $|mD|$ is base--point free. \\

Here is our analogue of the usual Kawamata--Viehweg vanishing theorem in characteristic zero. Our original statement only treated the big and semi--ample case. We thank Shunsuke Takagi for suggesting us this generalization.

\begin{thm}\label{Witt_KVV}
	Let $Y$ be a normal projective variety and let $D$ be an effective, semi--ample divisor on $Y$ such that $Y \setminus D$ is smooth. Then for all $i < \kappa(Y, D)$, we have \[ H^i(Y, WI_{D, \bQ}) = 0. \]
\end{thm}
\begin{rem}\label{rem_proper_proj}
	The proof shows that if $D$ is already big and semi--ample, then we only need $Y$ to be proper. This is probably still true in the non--big case, but we will not treat this question here.
\end{rem}

\begin{proof}
	Let $m \gg 0$ be a big and divisible enough integer so that the linear system $|mD|$ defines a proper morphism $\pi \colon Y \to X$ to a projective variety $X$, and let $A$ be an effective, Cartier and very ample divisor on $X$ such that $mD \sim \pi^*A$. Up to making $m$ bigger, we may assume that $\pi_*\cO_Y = \cO_X$. Then, the projection formula gives that pulling back induces an isomorphism $H^0(X, \cO_X(A)) \cong H^0(Y, \cO_Y(mD))$ so given that $mD$ is effective, there exists a Cartier, effective, very ample divisor $A' \sim A$ such that we have an \emph{equality} of divisors $mD = \pi^*A'$. 
	
	Note that $\dim(X) = \kappa(Y, D)$ by definition and that $X \setminus A'$ is affine, so $Y \setminus D$ satisfies the properties of \autoref{vanishing_rigid_cohomology_top_slopes}. In particular, we deduce that for all $i < \kappa(Y, D)$, we have \[ 0 = H^i_c(Y \setminus D , W\cO_{Y \setminus D, K}) \expl{=}{by definition, see \cite[Section 2.9]{Berthelot_Bloch_Esnault_On_Witt_vector_cohomology_for_singular_varieties}} H^i(Y, WI_D) \otimes_{W(k)} K. \]
	We then conclude by \autoref{finiteness_p_torsion} that also $H^i(Y, WI_{D, \bQ}) = 0$ for all $i < \kappa(Y, D)$.
%	If furthermore $D$ has simple normal crossings, then $Y$ is automatically smooth and for all $i > 0$, we have that \[ 0 \expl{=}{\autoref{vanishing_rigid_cohomology_top_slopes}} H^{d + i}_{\rig}((Y \setminus D)/K)_{[d]} \expl{\cong}{\cite[Theorems 2.4.4 and 3.1.1]{Shiho_crystalline_fundametal_groups_II}} H^{d + i}_{\log\crys}(Y/K)[1/p]_{[d]} \expl{\cong}{\JB{ref}} H^i(Y, W\omega_Y(\log D)_{\bQ}), \] so the proof is complete. \JB{ref for last iso is ] (A.) – Cohomologie cristalline des vari´et´es ouvertes, Rev. Maghr´ebine Math., t. 2 (1993), pp. 161–175., but I can't find a pdf}
\end{proof}

%\begin{rem}
%	As the proof above shows, the vanishing for $WI_D$ uses the vanishing of rigid cohomology for all the slopes between $d - 1$ and $d$, while the vanishing for $W\omega_Y(\log D)$ only used the vanishing of rigid cohomology in slope $d$. 
%	
%	This contrasts with the characteristic zero picture, where the analogues of these two vanishing are equivalent by Serre duality. The point is exactly that \emph{Serre duality fails} in our setup here. For example, if $E$ is a supersingular elliptic curve, then $H^1(E, W\cO_{E, \bQ})$ has dimension $2$ (with slope $1/2$), while $H^0(E, W\omega_{E, \bQ}) = 0$.
%	
%	\JB{Perhaps add something to say that it's topological in max slope}
%\end{rem}

\begin{defn}
	Let $X$ be a topological space, and let $\{\cF^{\bullet}_n\}_{n \geq 1}$ be a collection of objects in the derived category of abelian sheaves of $X$ on some site.
	\begin{itemize}
		\item  We say that the sheaves $\cF^{\bullet}_n$ are of \emph{uniform $p$--power torsion} if there exists $e > 0$ such that for all $n \geq 1$, the map \[ \cF_n^{\bullet} \xto{p^e} \cF_n^{\bullet} \] is zero.
		\item We say that a collection of maps $f_n \colon \cF^{\bullet}_n \to \cG_n^{\bullet}$ is an \emph{isomorphism up to uniform $p$--power torsion} if the objects $\Cone(f_n)$ are of uniform $p$--power torsion.
	\end{itemize}
\end{defn}
\begin{rem}\label{rem_unif_p_power}
	\begin{enumerate}
		\item If each $\cF_n^{\bullet}$ are actually sheaves, then this is equivalent to saying that for some $e > 0$, $p^e\cdot \cF_n^{\bullet} = 0$ for all $n \geq 1$.
		\item\label{derived_cat_lemma} Note that if there exist $a \leq b \in \bZ$ such that $\cH^i(\cF_n^{\bullet}) = 0$ for all $i < a$ or $i > b$, then $\{\cF_n^{\bullet}\}_{n \geq 1}$ is of uniform $p$--power torsion if and only if for all $i \in \bZ$, the sheaves $\{\cH^i(\cF_n^{\bullet})\}_{n \geq 1}$ are of uniform $p$--power torsion. Indeed, this is a consequence of the same method as in the proof of \cite[Claim 7.5]{Baudin_Bernasconi_Kawakami_Frobenius_GR_fails}.
	\end{enumerate}
\end{rem}

We only phrase the following result for big and semi--ample divisors, although it holds in the same generality as in \autoref{Witt_KVV}.

\begin{cor}\label{Witt_KVV_finite_level}
	Let $Y$ be a normal proper variety of dimension $d$, and let $D$ be an effective, big and semi--ample divisor on $Y$ such that $Y \setminus D$ is smooth. Then for all $i < d$, the groups $H^i(Y, W_nI_D)$ are of uniform $p$--power torsion. 
	
	If $Y$ is in addition Cohen--Macaulay, then the groups $H^i(Y, W_n\omega_{(Y, D)})$ are also of uniform $p$--power torsion.
\end{cor}
\begin{proof}
	By Grothendieck duality and \autoref{logarithmic_works_well_when_CM}, it is enough to show the first statement. Thanks to the short exact sequences \[ \begin{tikzcd}
		0 \arrow[rr] &  & F^n_*WI_D \arrow[rr, "V^n"] &  & WI_D \arrow[rr] &  & W_nI_D \arrow[rr] &  & 0,
	\end{tikzcd}\] the result is an immediate consequence of \autoref{Witt_KVV} (see \autoref{rem_proper_proj}) when $i \leq d - 2$. To see the case when $i = d - 1$, note that the image of $H^{d - 1}(Y, W_nI_D) \to H^d(Y, F^n_*WI_D)$ is $p^n$--torsion, and hence $p^e$--torsion for some $e > 0$ independent of $n$ by \autoref{finiteness_p_torsion}. Hence, the case $i = d - 1$ is also a consequence of \autoref{Witt_KVV}.
\end{proof} 

\section{Witt--Grauert--Riemenschneider vanishing}

\subsection{First reduction}

Throughout, fix an integer $d \geq 1$. We want to reduce the proof of Witt--GR vanishing from general varieties to projective ones, in order to be able to apply Witt--KVV. 

\begin{situation}\label{basic_setup}
	We have an alteration $\pi \colon Y \to X$, with $Y$ smooth of dimension $d$.
\end{situation}

\begin{notation}
	Define $($Witt-GR$_d)$ (resp. $(\bQ_p$-GR$_d))$ to be the following statement: in \autoref{basic_setup}, the sheaves $R^i\pi_*W_n\omega_Y$ (resp. $R^i\pi_*W_n\omega_Y^{\perf}$) are of uniform $p$--power torsion for all $i > 0$.
\end{notation}

\begin{lem}\label{name_Q_p_GR}
	In \autoref{basic_setup}, the following are equivalent:
	\begin{enumerate}
		\item\label{first} for all $i > 0$, $R^i\pi_*W\omega_{Y, \bQ} = 0$;
		\item\label{second} for all $i > 0$, the sheaves $R^i\pi_*W_n\omega_Y^{\perf}$ are of uniform $p$--power torsion.
	\end{enumerate}
\end{lem}
\begin{proof}
	Assume first \autoref{first}. By definition, this means that for some $e > 0$, we have that $p^e\cdot R^i\pi_*W\omega_Y = 0$ for all $i > 0$. Since \autoref{cone_of_p_Witt_omega} gives rise to long exact sequences \[ \begin{tikzcd}
		\dots  \arrow[r] & R^i\pi_*W\omega_Y \arrow[r, "p^n"] & R^i\pi_*W\omega_Y \arrow[r] & R^i\pi_*W_n\omega_Y^{\perf} \arrow[r] & R^{i + 1}\pi_*W\omega_Y \arrow[r, "p^n"] & \dots
	\end{tikzcd} \] we deduce that each $R^i\pi_*W_n\omega_Y^{\perf}$ is automatically $p^{2e}$-torsion.

	Now, assume \autoref{second}. Combining \cite[Lemma 4.4.2.(b)]{Baudin_Duality_between_Witt_Cartier_crystals_and_perverse_sheaves} and \autoref{general_Mittag_Leffler_property}, we deduce that each inverse system $\{R^i\pi_*W_n\omega_Y^{\perf}\}_{n \geq 1}$ satisfies the Mittag--Leffler condition. Hence, we deduce by \autoref{Witt_omega_Mittag_Leffler} that \[ R^i\pi_*W\omega_Y = \lim_n R^i\pi_*W_n\omega_Y^{\perf}, \] so this direction is immediate.
\end{proof}

\begin{rem}\label{explanation_name}
	We can now explain why we decided to call the vanishing of the sheaves $R^i\pi_*W\omega_{Y, \bQ}$ as $\bQ_p$--GR vanishing. By the duality theory developed in \cite{Baudin_Duality_between_Witt_Cartier_crystals_and_perverse_sheaves}, $\bQ_p$--GR vanishing is equivalent to the fact that the sheaves $R\pi_*(\bZ/p^n\bZ)[\dim Y]$ are perverse up to uniform $p$--power torsion. In particular, if a theory of perverse $\bQ_p$--sheaves analogous to the theory of perverse $\bQ_l$--sheaves of \cite{Beilinson_Bernstein_Deligne_Faisceaux_Pervers} was developed, then $\bQ_p$--GR should be equivalent to the fact that $R\pi_*\bQ_p[\dim Y]$ is a perverse sheaf. In particular, this vanishing is much more related to the sheaf $\bQ_p$ than to the bigger sheaf $W\cO$.
\end{rem}

\begin{lem}\label{Patakfalvi_Zdanowicz_criterion}
	Let $\pi \colon Y \to X$ be a proper morphism of varieties. Then the following are equivalent:
	\begin{enumerate}
		\item the natural map $W\cO_{X, \bQ} \to R\pi_*W\cO_{Y, \bQ}$ is an isomorphism;
		\item the natural maps $W_n\cO_Y \to R\pi_*W_n\cO_X$ are isomorphisms up to uniform $p$--power torsion.
	\end{enumerate}
\end{lem}
\begin{proof}
	This follows from \cite[Proposition 3.10]{Patakfalvi_Zdanowicz_Ordinary_varieties_with_trivial_canonical_bundle_are_not_uniruled} and \autoref{rem_unif_p_power}.\autoref{derived_cat_lemma}.
\end{proof}

We will need the following version of splitting of trace maps in the Witt setup (its inseparable case will be treated in \autoref{invariance_univ_homeo} and only be needed later on):

\begin{lem}\label{split_trace_separable}
	Let $f \colon Z \to S$ be a finite, surjective and separable morphism of normal varieties. Then there exists an integer $r > 0$ and a Cartier equivariant morphism of inverse systems $\{W_n\omega_S^{\bullet}\}_{n \geq 1} \to \{Rf_*W_n\omega_Z^{\bullet}\}_{n \geq 1}$ such that each diagram \[ \begin{tikzcd}
		W_n\omega_S^{\bullet} \arrow[rr] \arrow[rrrr, "\cdot r", bend left =20] &  & Rf_*W_n\omega_Z^{\bullet} \arrow[rr, "\Tr"] &  & W_n\omega_X^{\bullet}
	\end{tikzcd} \] commutes.
\end{lem}
\begin{proof}
	By taking the Galois closure, we may assume that the field extension induced by $f$ is Galois of group $G$. For all $n \geq 1$, we have the usual (Frobenius equivariant) trace map $\Tr \colon Rf_*W_n\cO_S = f_*W_n\cO_Z \to W_n\cO_S$ given by \[ s \in W_n\cO_Z \longmapsto \sum_{g \in G} g(s) \in W_n\cO_S, \] (where $G$ acts component-wise on $W_n\cO_S$). Hence, the diagram
	\[ \begin{tikzcd}
		W_n\cO_{S} \arrow[rr] \arrow[rrrr, "\cdot r", bend left = 20] &  & {Rf_*W_n\cO_Z} \arrow[rr, "\Tr"] &  & W_n\cO_S
	\end{tikzcd}  \] commutes, where $r = \deg(f)$. Applying Grothendieck duality concludes the proof.
\end{proof}

\begin{lem}\label{enough_to_show_in_proj_case}
	Assume that $($Witt-GR$_d)$ (resp. $(\bQ_p$-GR$_d)$) holds for alterations between projective varieties as in \autoref{basic_setup}. Then  $($Witt-GR$_d)$ (resp. $(\bQ_p$-GR$_d)$) holds for any morphism as in \autoref{basic_setup}.
\end{lem}

\begin{proof}
	Both $($Witt-GR$_d)$ and $(\bQ_p$-GR$_d)$ are local on the base, so we may assume that $X$ is quasi-projective. We may also assume that $X$ is integral, since $Y$ is. Hence, using Nagata's compactification theorem \stacksproj{0F41}, there exists a commutative square of integral schemes \[ \begin{tikzcd}
		\overline{Y} \arrow[rr, "\overline{\pi}"] &  & \overline{X}        \\
		Y \arrow[u, "\inc"] \arrow[rr, "\pi"]     &  & X \arrow[u, "\inc"]
	\end{tikzcd} \] where the vertical arrows are open immersions, $\overline{X}$ is projective and $\overline{Y}$ is proper. Note that by properness of $\pi$, we automatically have that $\overline{\pi}^{-1}(X) = Y$. Now, the idea is to proceed as in the end of the proof of \cite[Proposition 4.4.17]{Rulling_Chatzimatiaou_Hodge_Witt_cohomology_and_Witt_rational_singularities}.
	
	By \cite[Remark 4.3.2]{Rulling_Chatzimatiaou_Hodge_Witt_cohomology_and_Witt_rational_singularities} (see also \cite{de_Jong_smoothness_semistability_and_alterations, de_Jong_Families_of_curves_and_alterations}), there exist morphisms \[ \overline{Z} \xto{\phi_1} \overline{S} \xto{\phi_2} \overline{T} \xto{\phi_3} \overline{Y}, \] where
	
	\begin{itemize}
		\item $\phi_1$ is finite and induces a Galois field extension, $\phi_2$ is projective and birational and $\phi_3$ is finite and induces a purely inseparable field extension.
		\item $\overline{Z}$ is smooth projective, and both $\overline{S}$ and $\overline{T}$ are normal and proper.
	\end{itemize}
	
	Set $\phi \coloneqq \phi_1 \circ \phi_2 \circ \phi_3$ and $Z \coloneqq \phi^{-1}(Y)$ (similarly define $S$ and $T$). By \autoref{split_trace_separable}, we have a commutative diagram
	\[  \begin{tikzcd}
		W_n\omega_S^{\bullet} \arrow[rr] \arrow[rrrr, "\cdot r", bend left =20] &  & R\phi_{1, *}W_n\omega_Z^{\bullet} \arrow[rr, "\Tr"] &  & W_n\omega_X^{\bullet}
	\end{tikzcd} \] for some integer $r > 0$. Pushing all this to $Y$ gives a diagram \[ \begin{tikzcd}
	R(\phi_3 \circ \phi_2)_*W_n\omega_S^{\bullet} \arrow[r] \arrow[rr, "\cdot r", bend left=20] & R\phi_*W_n\omega_Z^{\bullet} \arrow[r] & R(\phi_3 \circ \phi_2)_*W_n\omega_S^{\bullet} \arrow[r, "\Tr"] & W_n\omega_Y^{\bullet}
	\end{tikzcd} \] 
	Note that the maps $W_n\cO_Y \to R(\phi_3 \circ \phi_2)_*W_n\cO_S$ are isomorphisms up to uniform $p$--power torsion by \cite[Theorem 2]{Rulling_Chatzimatiaou_Hodge_Witt_cohomology_and_Witt_rational_singularities} and \autoref{Patakfalvi_Zdanowicz_criterion}, so the same holds for the trace maps $R(\phi_2 \circ \phi_3)_*W_n\omega_S^{\bullet} \to W_n\omega_Y^{\bullet}$ by Grothendieck duality. Applying the functor $R\pi_*[-d]$ gives a diagram
	\[\begin{tikzcd}
		C_n^{\bullet} \arrow[r] \arrow[rr, "\cdot r", bend left = 20] & R(\pi \circ \phi)_*W_n\omega_Z \arrow[r] & {C_n^{\bullet}} \arrow[r, "\Tr"] & R\pi_*W_n\omega_Y,
	\end{tikzcd} \] where $C_n^{\bullet} \coloneqq R(\pi \circ \phi_3 \circ \phi_2)_*W_n\omega_S^{\bullet}[-d]$ (recall that $Z$ and $Y$ are in particular Cohen--Macaulay, so their $W_n$--dualizing complex in only supported in degree $-d$). Taking $\Rlim_n$ and $(\cdot)_{\bQ}$ above shows that $R\pi_*W\omega_{Y, \bQ}$ is a direct summand of $R(\pi \circ \phi)_*W\omega_{Z, \bQ}$, so $\bQ_p$-GR vanishing for $\overline{Z} \to \overline{X}$ immediately implies $\bQ_p$-GR vanishing for $Y \to X$ by \autoref{name_Q_p_GR} (one could also mimic the upcoming argument via the duality theory from \cite{Baudin_Duality_between_Witt_Cartier_crystals_and_perverse_sheaves}). 

	Now, assume that for all $i > 0$, the sheaves $R^i(\pi \circ \phi)_*W_n\omega_Z$ are of uniform $p$--power torsion. Since we have a diagram 
	\[ \begin{tikzcd}
		\cH^i(C_n^{\bullet}) \arrow[r] \arrow[rr, "\cdot r", bend left = 20] & R^i(\pi \circ \phi)_*W_n\omega_Z \arrow[r] & \cH^i(C_n^{\bullet})
	\end{tikzcd} \] we deduce that also all $\cH^i(C_n^{\bullet})$ with $i > 0$ are also of uniform $p$--power torsion. Since the maps $\cH^i(C_n^{\bullet}) \to R^i\pi_*W_n\omega_Y$ are isomorphisms up to uniform $p$--power torsion as already explained, the proof is complete.
\end{proof}

\subsection{Witt-GR vanishing for threefolds and for isolated singularities}

We start with the proof of Witt-GR vanishing for threefolds, because we believe that it sheds light on the proof of $\bQ_p$-GR vanishing, and especially why it works in full generality.

\begin{thm}\label{most_general_Witt_GR}
	Let $\pi \colon Y \to X$ be an alteration, with $Y$ smooth and $X$ projective. Assume that for all $i > 0$, the sheaf $R^i\pi_*\omega_Y$ is supported at points. Then there exists $e > 0$ such that for all $i > 0$ and $n \geq 1$, \[ p^e\cdot R^i\pi_*W_n\omega_Y = 0. \]
\end{thm}
\begin{rem}
	With the appropriate version of existence of resolution of singularities in positive characteristic and \cite[Theorem 2]{Rulling_Chatzistamatiou_GR_for_smooth}, one should be able to remove the projectivity hypothesis on $X$ in the theorem above and in \autoref{Witt_GR_isolated_singularities}.
\end{rem}
\begin{proof}
	Let $A$ be a very ample effective divisor on $X$ disjoint from $\Supp(R^i\pi_*\omega_Y)$ for all $i > 0$, such that for all $j > 0$ and $i \geq 0$, $H^j(X, R^i\pi_*\omega_Y \otimes \cO_X(A)) = 0$. Set $D \coloneqq \pi^*A$. 
	
	Since each $R^i\pi_*\omega_Y(D) \cong R^i\pi_*\omega_Y \otimes \cO_X(A)$ with $i > 0$ is supported at finitely many points by assumption, we deduce that it is also the case of each sheaf $R^i\pi_*W_n\omega_{(Y, D)}$ with $i > 0$ and $n \geq 1$ by \autoref{logarithmic_works_well_when_CM}. 
	
	We claim that for all $i \geq 0$ and $j > 0$, the groups \[ H^j(X, R^i\pi_*W_n\omega_{(Y, D)}) \] are of uniform $p$--power torsion.
	
	If $i > 0$, these groups actually vanish, since a sheaf supported in dimension zero has no higher cohomology groups. If $i = 0$ and $j = 1$, then \[ H^1(X, \pi_*W_n\omega_{(Y, D)}) \inc H^1(Y, W_n\omega_{(Y, D)}), \] which is of uniform $p$--power torsion by \autoref{Witt_KVV_finite_level}. Finally, we will show by induction on $n \geq 1$ that for all $j \geq 2$, $H^j(X, \pi_*W_n\omega_{(Y, D)}) = 0$. The case $n = 1$ follows from the choice of $A$, since $\pi_*W_1\omega_{(Y, D)} = \pi_*\omega_Y(D) \cong \pi_*\omega_Y \otimes \cO_X(A)$. For $n \geq 2$, consider the exact sequence \[ 0 \to W_{n - 1}\omega_{(Y, D)} \to W_n\omega_{(Y, D)} \to F^{n - 1}_*\omega_Y(D) \to 0, \] (see \autoref{logarithmic_works_well_when_CM}), and define the sheaves 
	\begin{align*}
		&\cI_n \coloneqq \im(\pi_*W_n\omega_{(Y, D)} \to \pi_*F^{n - 1}_*\omega_Y(D)) \\
		&\cC_n \coloneqq \ker(R^1\pi_*W_{n - 1}\omega_{(Y, D)} \to R^1\pi_*W_n\omega_{(Y, D)}).
	\end{align*} Note that by construction, we have an exact sequence \[ 0 \to \cI_n \to \pi_*F^{n - 1}_*\omega_Y(D) \to \cC_n \to 0. \] Then by induction, we have that for $i \geq 2$, \[ H^i(X, \pi_*W_n\omega_{(Y, D)}) \cong H^i(X, \cI_n) \cong H^{i - 1}(X, \cC_n) \expl{=}{$\cC_n$ is supported in dimension zero} 0. \]
	Thus, our claim is proven. It automatically implies that for all $i > 0$, \[ H^0(X, R^i\pi_*W_n\omega_{(Y, D)}) \cong H^i(Y, W_n\omega_{(Y, D)})\] up to uniform $p$--power torsion, so we conclude by \autoref{Witt_KVV_finite_level} that for some $e > 0$, \[ p^e \cdot H^0(X, R^i\pi_*W_n\omega_{(Y, D)}) = 0. \] Since the sheaves $R^i\pi_*W_n\omega_{(Y, D)}$ are supported in dimension zero, this shows that $p^e \cdot R^i\pi_*W_n\omega_{(Y, D)} = 0$, so in particular $p^e \cdot R^i\pi_*W_n\omega_Y|_{X \setminus A} = 0$. Since $A$ avoids the support of each $R^i\pi_*W_n\omega_Y$, the proof is complete.
\end{proof}

To unconditionally deduce our vanishing theorem for threefolds, we will need the following vanishing result, which is a consequence of GR--vanishing on surfaces in arbitrary characteristic:

\begin{lem}\label{GR_surfaces}
	Let $\pi \colon Y \to X$ be an alteration with $Y$ smooth. Then for all $i > 0$, \[ \codim \Supp(R^i\pi_*\omega_Y) \geq i + 2. \]
\end{lem}
\begin{proof}
	This is exactly \cite[Proposition 2.6.1]{Hacon_Pat_GV_Characterization_Ordinary_AV} (see the arXiv version).
\end{proof}

\begin{cor}\label{Witt_GR_threefolds}
	Let $\pi \colon Y \to X$ be an alteration, with $Y$ smooth of dimension at most $3$. Then there exists $e > 0$ such that for all $i > 0$ and $n \geq 1$, \[ p^e\cdot R^i\pi_*W_n\omega_Y = 0. \]
\end{cor}

\begin{proof}
	We may assume that $Y$ and $X$ are projective by \autoref{enough_to_show_in_proj_case}. By \autoref{GR_surfaces}, the hypothesis of \autoref{most_general_Witt_GR} is satisfied, so we are done.
\end{proof}

\begin{cor}\label{Witt_GR_isolated_singularities}
	Let $\pi \colon Y \to X$ be a projective birational morphism of varieties with $Y$ smooth. Assume furthermore that $X$ is projective and has isolated singularities. Then there exists $e > 0$ such that for all $i > 0$ and $n \geq 1$, \[ p^e\cdot R^i\pi_*W_n\omega_Y = 0. \]
\end{cor}
\begin{proof}
	This is an immediate consequence of \cite[Theorem 2]{Rulling_Chatsistamatiaou_Higher_direct_image_in_positive_characteristic} and \autoref{most_general_Witt_GR}.
\end{proof}

%\begin{rem}\label{problems_proof_Witt_GR}
%	If one wanted to show this result for fourfolds, this proof seems to break down at two different places. The problem is about killing the cohomology groups $H^j(X, R^i\pi_*W\omega_X(\log D))$. 
%	Say we want to prove it in dimension four. Then we can conclude as in \JB{ref Hacon/Zsolt} that $R^2\pi_*W\omega_Y$ and $R^2\pi_*W\omega_Y(\log D)$ are of bounded $p$--power torsion. Furthermore, again by induction (but this time on $D$), we know that we have a surjection $R^1\pi_*W\omega_Y \to R^1\pi_*W\omega_Y(\log D)$ up to bounded $p$--power torsion.
%	
%	If $n = 1$ and this was actually a surjection, we would automatically deduce that $R^1\pi_*\omega_Y$ can only be supported in dimension zero. The problem is that we do not know how to show that $R^1\pi_*W_n\omega_Y$ is supported in dimension zero (up to bounded $p$--power torsion). The issue seems to be that we do not know if a single very ample divisor $A \inc X$ can give positivity to all the sheaves $R^1\pi_*W_n\omega_Y(\log D)$ (and the absence of a projection formula in this context). 
%	
%	Even if we could prove that, it is nevertheless unclear how to show that the groups $H^{\geq 2}(X, \pi_*W_n\omega_X(\log D))$ also are of bounded $p$--power torsion. The approach would again not work.
%\end{rem}

\subsection{Proof of $\bQ_p$--GR vanishing}

Throughout, fix an alteration $\pi \colon Y \to X$ with $Y$ smooth (i.e. we are in \autoref{basic_setup}). By \autoref{name_Q_p_GR}, our goal is to show that for $i > 0$, the sheaves \[ R^i\pi_*W_n\omega_Y^{\perf} = 0 \] are of uniform $p$--power torsion. By \cite[Lemma 4.4.2]{Baudin_Duality_between_Witt_Cartier_crystals_and_perverse_sheaves}, we then need to investigate the Witt--Cartier \emph{crystals} $R^i\pi_*W_n\omega_Y$. 

\begin{lem}\label{everyone_has_the_same_simple_subquot}
	For any $i \geq 0$ and $n \geq 1$, the set of simple subquotients of the $W_n$--Cartier crystal $R^i\pi_*W_n\omega_Y$ (see \autoref{crystals_have_finite_length}) is contained in that of $R^i\pi_*\omega_Y$.
\end{lem}
\begin{proof}
	We have a long exact sequence \[ \dots \to R^i\pi_*W_n\omega_Y \to R^i\pi_*W_{n + 1}\omega_Y \to F^n_*R^i\pi_*\omega_Y \to \dots. \] Since $F^n_*R^i\pi_*\omega_Y \sim_C R^i\pi_*\omega_Y$ (the Cartier module action is always a nil--isomorphism), it follows immediately that a simple subquotient of $R^i\pi_*W_{n + 1}\omega_Y$ is also a simple subquotient of either $R^i\pi_*W_n\omega_Y$ or $R^i\pi_*\omega_Y$. We then conclude the proof by induction on $n \geq 1$.
\end{proof}

\begin{lem}\label{the_golden_neighborhood}
	Let $x \in X$. Then there exists an affine open neighborhood $U$ of $x$ with the following property: for all $i \geq 0$, $n \geq 1$ and $e > 0$, then \[ p^e \cdot H^0(U, R^i\pi_*W_n\omega_Y) \sim_C 0 \: \implies \: p^e \cdot R^i\pi_*W_n\omega_Y|_U \sim_C 0. \]
\end{lem}
\begin{rem}\label{the_golden_neighborhood_is_even_more_golden}
	The proof shows that $U$ can be taken to be an arbitrarily small open neighborhood of $x$.
\end{rem}
\begin{proof}
	Let $S$ denote the collection of simple subquotients of all the crystals $R^i\pi_*W_n\omega_Y$ (up to isomorphism), with $i \geq 0$ and $n \geq 1$. Then $S$ is finite by \autoref{everyone_has_the_same_simple_subquot} and \autoref{crystals_have_finite_length}. Fix $\cM \in S$. By \cite[Lemma 4.4.10]{Baudin_Duality_between_Witt_Cartier_crystals_and_perverse_sheaves}, there exists an open neighborhood $U_{\cM}$ of $x$ such that if $\cM_x \not\sim_C 0$, then $H^0(U_{\cM}, \cM) \not\sim_C 0$. Since $\cM$ is torsion--free on its support, we may replace $U_{\cM}$ by any smaller affine neighbourhood of $x$, and hence further assume that $\cM_x \sim_C 0$ if and only if $\cM|_{U_{\cM}} \sim_C 0$. Since $S$ is finite, we may assume that $U \coloneqq U_{\cM} = U_{\cM'}$ for all $\cM, \cM' \in S$.
	
	Let us show that $U \ni x$ satisfies the conclusion of the lemma, so assume that for some $i \geq 0$, $n \geq 1$ and $e > 0$, we have that $p^e \cdot H^0(U, R^i\pi_*W_n\omega_Y) \sim_C 0$. Consider the exact sequence \[ 0 \to R^i\pi_*W_n\omega_Y[p^e] \to R^i\pi_*W_n\omega_Y \to \cQ \to 0. \] Since $U$ is affine, the functor $H^0(U, -)$ is exact, so we obtain from the assumption that $H^0(U, \cQ) \sim_C 0$. Again by affinity of $U$, we obtain that this is true for all simple subquotients of $\cQ$. Since all these simple subquotients belong to $S$, we deduce that $\cQ|_U \sim_C 0$ by definition of $U$. This gives us that $R^i\pi_*W_n\omega_Y[p^e]|_U \sim_C R^i\pi_*W_n\omega_Y|_U$, or in other words that $p^e \cdot R^i\pi_*W_n\omega_Y|_U \sim_C 0$.
\end{proof}

We can finally prove $\bQ_p$--GR vanishing.
\begin{thm}\label{Qp_GR}
	Let $\pi \colon Y \to X$ be an alteration with $Y$ smooth. Then for all $i > 0$, we have the vanishing \[ R^i\pi_*W\omega_{Y, \bQ} = 0. \]
\end{thm}

\begin{proof}
	By \autoref{enough_to_show_in_proj_case}, we may assume that $Y$ and $X$ are projective. Fix $x \in X$, and let $U$ be an affine open neighborhood of $x$ satisfying the conditions of \autoref{the_golden_neighborhood}. By \autoref{the_golden_neighborhood_is_even_more_golden}, we may assume that its complement $A = X \setminus U$ is an ample Cartier divisor. Set $D \coloneqq \pi^*A \subseteq Y$ (which is then big and semi--ample), and $V \coloneqq Y \setminus D = \pi^{-1}(U)$. We will also write $\pi$ for the induced proper map $V \to U$. By \autoref{pushforward_from_open}, we know that \[ W_n\omega_{(Y, D)} \sim_C Rj_{V, *}W_n\omega_{V} \] (since $D$ is Cartier, $j_V$ is an affine map, so higher pushforwards vanish). Thus, \[ R\pi_*W_n\omega_{(Y, D)} \sim_C Rj_{U, *}R\pi_*W_n\omega_{V} \] and since $j_U$ is also an affine morphism, we deduce that for all $i \geq 0$, \[  
	R^i\pi_*W_n\omega_{(Y, D)} \sim_C j_{U, *}R^i\pi_*W_n\omega_{V} = Rj_{U, *}R^i\pi_*W_n\omega_{V}. \]
	This shows that for all $j > 0$, \[ H^j(X, R^i\pi_*W_n\omega_{(Y, D)}) \sim_C H^j(U, R^i\pi_*W_n\omega_Y) \expl{=}{$U$ is affine} 0. \]
	Hence, the Leray spectral sequence shows that for all $i > 0$, \[ H^i(Y, W_n\omega_{(Y, D)}) \sim_C H^0(X, R^i\pi_*W_n\omega_{(Y, D)}) \expl{\sim_C}{\autoref{pushforward_from_open}} H^0(U, R^i\pi_*W_n\omega_{V}). \]
	By \autoref{Witt_KVV_finite_level}, there exists $e > 0$ such that $p^e \cdot H^i(Y, W_n\omega_{(Y, D)}) = 0$ for all $i > 0$ and $n \geq 1$, so by the choice of $U = X \setminus A$ (see \autoref{the_golden_neighborhood}), we deduce that \[ p^e \cdot (R^i\pi_*W_n\omega_Y)|_U = p^e \cdot R^i\pi_*W_n\omega_{V} \sim_C 0. \] Since the sheaf on the left--hand side is coherent on $U$ and locally nilpotent, its perfection is zero. Thus, we deduce that \[ \left(p^e \cdot R^i\pi_*W_n\omega_Y^{\perf}\right)\big|_U \expl{=}{\cite[Lemma 4.4.2.(b)]{Baudin_Duality_between_Witt_Cartier_crystals_and_perverse_sheaves}} \left(p^e \cdot (R^i\pi_*W_n\omega_Y)|_U\right)^{\perf} = 0. \] We then conclude by quasi--compactness of $X$ that the sheaves $R^i\pi_*W_n\omega_Y^{\perf}$ are of uniform $p$--power torsion, so the proof is complete by \autoref{name_Q_p_GR}.
\end{proof}

Recall that $W_n\omega_{Y, \log}$ denotes the étale $\bZ/p^n$--sheaf generated by local sections of the form $d\log([y_1]) \wedge \dots \wedge d\log([y_d])$, with $y_i \in \cO_Y^{\times}$. Recall also the definition of the weight $j$ motivic complex $\bZ/p^n(j)_Y$ (see e.g. \cite{Geisser_Levine_The_K_theory_of_fields_in_char_p}). As a consequence of \autoref{Qp_GR}, we obtain a vanishing result for Witt--logarithmic sheaves and top weight motivic complex (this statement was kindly suggested to us by Kay Rülling):

\begin{cor}\label{motivic_vanishing}
	Let $\pi \colon Y \to X$ be an alteration with $Y$ smooth of dimension $d$. Then for all $i > 0$, the sheaves $R^i\pi_*W_n\omega_{Y, \log}$ are of uniform $p$--power torsion and for all $i > d$, the sheaves $R^i\pi_*(\bZ/p^n(d)_Y)$ are also of uniform $p$--power torsion.
\end{cor}
\begin{proof}
	By \cite[Theorem 8.3]{Geisser_Levine_The_K_theory_of_fields_in_char_p}, we have that $\bZ/p^n(d)_Y = W_n\omega_{Y, \log}[-d]$, so it is enough to show the first statement. 
	
	For a $W_n$--Cartier module $(\cM, \theta)$ on $Y$, let us denote by $\Sol(\cM)$ the étale $\bZ/p^n$--sheaf defined by $\ker(\theta^{\et} - 1)$, where $\theta^{\et}$ denotes the extension of $\theta$ to the étale site of $X$ (see \cite[Definition 4.4.15]{Baudin_Duality_between_Witt_Cartier_crystals_and_perverse_sheaves}). By \cite[Proposition 3.4]{Kato_Duality_theories_for_p_primary_etale_coh_II} (see also \cite{Ekedahl_Duality_Hodge_Witt}), we know that $W_n\omega_{Y, \log} = \Sol(W_n\omega_Y)$. Since $\Sol$ is exact and preserves derived pushforwards (\cite[Corollary 4.4.18]{Baudin_Duality_between_Witt_Cartier_crystals_and_perverse_sheaves}), we know that $\Sol(p^e\cdot R^i\pi_*W_n\omega_Y) = p^e\cdot R^i\pi_*W_n\omega_{Y, \log}$. By \autoref{Qp_GR} and \autoref{name_Q_p_GR}, we know that there exists $e > 0$ such that $p^e \cdot R^i\pi_*W_n\omega_Y \sim_C 0$ for all $i > 0$ and $n \geq 1$, so we conclude the proof by \cite[Corollary 4.4.18]{Baudin_Duality_between_Witt_Cartier_crystals_and_perverse_sheaves}.
\end{proof}

\section{First applications to singularities}

\subsection{Around Witt--rational and $\bQ_p$--rational singularities}\label{boring_section}

We follow the same notations as \cite[0.2]{Rulling_Chatzimatiaou_Hodge_Witt_cohomology_and_Witt_rational_singularities}. Namely, we say that a variety $Y$ is a finite quotient if it is normal and admits a finite surjective morphism from a smooth variety, and that it is a topologically finite quotient if it admits a finite universal homeomorphism onto a finite quotient. Finally, we say that a morphism $f \colon Y \to X$ between two varieties is a quasi--resolution if $Y$ is a topologically finite quotient and $f$ is surjective, projective, generically finite and purely inseparable (any variety admits a quasi--resolution by results of de Jong \cite{de_Jong_Families_of_curves_and_alterations, de_Jong_smoothness_semistability_and_alterations}).

For the formalism of étale $\bZ_p$--sheaves and $\bQ_p$--sheaves that we use, we refer the reader to \cite[Definition 3.14]{Patakfalvi_Zdanowicz_Ordinary_varieties_with_trivial_canonical_bundle_are_not_uniruled}.

\begin{defn}
	Let $X$ be an irreducible variety. 
	\begin{itemize}
		\item We say that $X$ has \emph{$W\cO$--rational} (resp. \emph{$\bQ_p$--rational}) \emph{singularities} if for any quasi--resolution $\pi \colon Y \to X$, we have $R^i\pi_*W\cO_{Y, \bQ} = 0$ (resp. $R^i\pi_*\bQ_{p, Y} = 0$) for all $i > 0$ and $W\cO_{X, \bQ} = \pi_*W\cO_{Y, \bQ}$ (resp. $\bQ_{p, X} = \pi_*\bQ_{p, Y}$).
		
		\item We say that $X$ has \emph{Witt--rational singularities} if it has $W\cO$--rational singularities, and Witt--GR vanishing holds for any quasi--resolution $\pi \colon Y \to X$.
		
		\item We say that $X$ has \emph{BE--$\bQ_p$--rational} singularities if for any alteration $\pi \colon Y \to X$, the natural map $\bQ_{p, X} \to R\pi_*\bQ_{p, Y}$ splits in the derived category of $\bQ_{p, X}$--modules.
	\end{itemize}
\end{defn}

\begin{rem}
	\begin{itemize}
		\item Our definition of Witt--rational singularities is (a priori) stronger than that defined in \cite{Rulling_Chatzimatiaou_Hodge_Witt_cohomology_and_Witt_rational_singularities}. We believe this definition is the right one, as theirs is fundamentally more related to $\bQ_p$--rationality than to $W\cO$--rationality (see \autoref{name_Q_p_GR} and \autoref{explanation_name}).
		\item Our definition of $BE$--$\bQ_p$--rational singularities mimics the Witt vector version defined in \cite{Blickle_Esnault_Rational_Singularities_and_rational_points} (see also \cite[Definition 4.4.2]{Rulling_Chatzimatiaou_Hodge_Witt_cohomology_and_Witt_rational_singularities}). We leave the study of $BE$--Witt--rationality for future work.
	\end{itemize}
\end{rem}

Let us add two more definitions, imitating the notions of Cohen--Macaulayness and pseudo--rationality:

\begin{defn}\label{def_Witt-CM_and_so_on}
	\begin{itemize}
		\item A variety $X$ is said to be \emph{Witt--Cohen--Macaulay}, (resp. \emph{$\bQ_p$--Cohen--Macaulay}), if for all $i \neq -\dim X$, it holds that the modules $\cH^i(W_n\omega_X^{\bullet})$ (resp. $\cH^i(W_n\omega_X^{\bullet})^{\perf}$) are of uniform $p$--power torsion. Often, we will shorten the notation to Witt--CM and $\bQ_p$--CM. 
		
		\item An irreducible variety $X$ is said to be \emph{Witt--pseudo--rational} (resp. \emph{$\bQ_p$--pseudo--rational}) if it is Witt--CM (resp. $\bQ_p$--CM) and for any alteration $\pi \colon Y \to X$, the cokernels of the natural maps $\pi_*W_n\omega_Y \to W_n\omega_X$ (resp. $\pi_*W_n\omega_Y^{\perf} \to W_n\omega_X^{\perf}$) are of uniform $p$--power torsion.
	\end{itemize}
\end{defn}
Here is a diagram explaining the different links between all these notions that we prove here, as well as the notions of $\bF_p$--rationality and $\bF_p$--Cohen--Macaulayness defined in \autoref{def_F_p_notions}:

\[ \begin{tikzcd}
	& \fbox{BE--$\bQ_p$--rational} \arrow[d, Leftrightarrow] & \fbox{$W\cO$--rational} \arrow[ld, Rightarrow] \arrow[d, Rightarrow, bend left=60, "\:\:{\scriptsize \left(\substack{\mbox{if Witt--GR holds,} \\ \mbox{e.g. in dimension } \leq 3}\right)}"{pos=0.6}]                             \\
	\fbox{$\bF_p$--rational} \arrow[r, Rightarrow] \arrow[dd, Leftrightarrow, "\times" {description, font = \large}"] & \fbox{$\bQ_p$--rational} \arrow[d, Leftrightarrow] \arrow[l, Rightarrow, bend right=20, "\times" {description, font = \large}] \arrow[r, Rightarrow, bend right=20, "\times" {description, font = \large} ]                     & \fbox{Witt--rational} \arrow[u, Rightarrow] \arrow[l, Rightarrow]                        \\
	& \fbox{$\bQ_p$--pseudo--rational} \arrow[d, Rightarrow]                                                                           & \fbox{Witt--pseudo--rational} \arrow[d, Rightarrow] \arrow[u, Leftrightarrow, "\: \: {\scriptsize \left(\substack{\mbox{if Witt--GR holds,} \\ \mbox{e.g. in dimension } \leq 3}\right)}"'{pos=0.4}] \arrow[l, Rightarrow] \\
	\fbox{$\bF_p$--CM} \arrow[r, Rightarrow]                              & \fbox{$\bQ_p$--CM} \arrow[u, Rightarrow, bend right=60, "\times" {description}] \arrow[r, Rightarrow, bend right=20, "\times" {description, font = \large}] \arrow[l, Rightarrow, bend left=20, "\times" {description, font = \large}] & \fbox{Witt--CM} \arrow[u, Rightarrow, bend right=60, "\times" {description}] \arrow[l, Rightarrow]                                    
\end{tikzcd}\]

where:

\begin{itemize}
	\item the links between $\bF_p$, $\bQ_p$ and $W\cO$--rationality are discussed in \cite[Lemma 3.16 and Example 3.17]{Baudin_Bernasconi_Kawakami_Frobenius_GR_fails};
	\item the fact that $\bF_p$--rational does not imply $\bF_p$-CM follows from \cite[Theorem 5.2]{Baudin_Bernasconi_Kawakami_Frobenius_GR_fails};
	\item the fact that Witt--CM (resp. Witt--pseudo--rational) implies $\bQ_p$--CM (resp. $\bQ_p$--pseudo--rational) follows by definition, and a counterexample for the converse is given by the cone over a supersingular abelian surface $A$ (it is $\bQ_p$--CM since it is $\bQ_p$--rational by proper base change, but it can be shown explicitly that it is not Witt--CM, essentially because $H^1(A, W\cO_{A, \bQ}) \neq 0$);
	\item the fact that $\bF_p$--CM implies $\bQ_p$--CM follows from the same argument as in the first paragraph of the proof of \autoref{logarithmic_works_well_when_CM};
	\item the fact that neither $\bF_p$--CM, $\bQ_p$--CM nor Witt--CM imply their rationality counterpart follow from taking a cone over an ordinary elliptic curve (such a cone is not $\bQ_p$--rational by proper base change);
	\item the links between the different $\bQ_p$--rationality notions (resp. Witt--rationality notions) follow from \autoref{thm_notions_of_singularities} (and its proof for the Witt notions).
\end{itemize}
Note also that usual rationality implies all these rationality notions defined above, and similarly for Cohen--Macaulayness. \\

Our main objective here is to show the following:

\begin{thm}\label{thm_notions_of_singularities}
	Let $X$ be an irreducible variety. Then the following are equivalent:
	\begin{enumerate}
		\item $X$ has $\bQ_p$--rational singularities;
		\item $X$ has $\bQ_p$--pseudo--rational singularities;
		\item $X$ has BE--$\bQ_p$--rational singularities.
	\end{enumerate}
	
	If $\dim(X) \leq 3$, then also the following are equivalent:
	\begin{enumerate}
		\item\label{one} $X$ has $W\cO$--rational singularities;
		\item\label{three} $X$ has Witt--pseudo--rational singularities.
		\item\label{four} $X$ has Witt--rational singularities;
	\end{enumerate}
\end{thm}

In order to prove these equivalences, we will need a few preliminary results:

\begin{lem}\label{invariance_univ_homeo}
	Let $f \colon Y \to X$ be a finite universal homeomorphism of varieties. Then the natural maps  $W_n\cO_X \to Rf_*W_n\cO_Y$ and $Rf_* W_n\omega_Y^{\bullet} \to W_n\omega_X^{\bullet}$ are isomorphisms up to uniform $p$--power torsion.
\end{lem}
\begin{proof}
	By Grothendieck duality, it is enough to show that the kernel and cokernel of $W_n\cO_X \to f_*W_n\cO_Y$ are of uniform $p$--power torsion. By \stacksproj{0CNF}, they are annihilated by $F^e$ for $e \gg 0$ in the case $n = 1$. But then, the same integer $e$ works for any $n$ by definition, so we are done (recall that $p = FV$).
\end{proof}

\begin{lem}\label{top_fin_quot_is_Witt_CM}
	A topologically finite quotient has Witt--CM singularities.
\end{lem}
\begin{proof}
	By \autoref{invariance_univ_homeo}, it is enough to show the result for a finite quotient, so let $f \colon Z \to Y$ be a finite surjective morphism with $Z$ smooth and $Y$ normal. Factor this morphism as $Z \to W \to Y$, where $W$ is normal, $g \colon W \to Y$ is finite separable, and $Z \to W$ is finite purely inseparable. Again by \autoref{invariance_univ_homeo}, $W$ is Witt--CM. We deduce that $Y$ is also Witt--CM by \autoref{split_trace_separable}.
%	To conclude that $Y$ is also Witt--CM, we can now use the same trace trick as in the proof of \autoref{enough_to_show_in_proj_case} (to obtain the trace map as in \emph{loc. cit.}, take the integral closure of $W$ in the Galois closure of $K(W)/K(Y)$). \JB{can be helped by an additional lemma!}
\end{proof}

\begin{lem}[{\cite{Rulling_Chatzimatiaou_Hodge_Witt_cohomology_and_Witt_rational_singularities}}]\label{enough_to_verify_easy_version}
	Let $X$ be an irreducible variety. Then if $X$ satisfies the $\bQ_p$--rationality (resp. $W\cO$--rationality, Witt--rationality) condition for one quasi--resolution, it satisfies it for any quasi--resolution.
\end{lem}
\begin{proof}
	It is enough to show that if $\pi \colon S \to T$ is any quasi--resolution with $S$ and $T$ topologically finite quotients, then it satisfies Witt--GR vanishing, and all sheaves $R^i\pi_*W\cO_{S, \bQ}$ and $R^i\pi_*\bQ_{p, S}$ vanish for $i > 0$. The fact that each $R^i\pi_*W\cO_{S, \bQ}$ vanishes is \cite[Theorem 3]{Rulling_Chatzimatiaou_Hodge_Witt_cohomology_and_Witt_rational_singularities}, and the analogue for $\bQ_p$ then follows from \cite[Lemma 3.19]{Patakfalvi_Zdanowicz_Ordinary_varieties_with_trivial_canonical_bundle_are_not_uniruled}. We are therefore left to show that Witt--GR holds for $\pi$. Since the maps $W_n\cO_T \to R\pi_*W_n\cO_S$ are isomorphisms up to uniform $p$--power torsion (see \autoref{Patakfalvi_Zdanowicz_criterion}), we obtain by Grothendieck duality that also $R\pi_*W_n\omega_S^{\bullet} \to W_n\omega_T^{\bullet}$ are isomorphisms up to uniform $p$--power torsion. Given that both $S$ and $T$ are Witt--CM (\autoref{top_fin_quot_is_Witt_CM}), we deduce that the sheaves $R^i\pi_*W_n\omega_S$ are also of uniform $p$--power torsion for $i > 0$.
\end{proof}

\begin{prop}\label{GR_for_top_finite_quot}
	Let $\pi \colon Y \to X$ be an alteration, where $Y$ is a topologically finite quotient. Then $\bQ_p$--GR holds for $\pi$. If furthermore $\dim(X) \leq 3$, then also Witt--GR holds.
\end{prop}
\begin{proof}
	Let $f \colon Y \to S$ be a finite universal homeomorphism, where $S$ is a finite quotient. Then for some $e > 0$, $\cO_Y^{p^e} \inc \cO_S$ by \stacksproj{0CNF}. Hence, there exists a diagram \[ \begin{tikzcd}
		S_{p^e} \arrow[r, "g"] \arrow[rr, "(F')^e"', bend right] & Y \arrow[r, "f"] & S,
	\end{tikzcd} \] where $(F')^e \colon S_{p^e} \to S$ denotes the $k$--linear Frobenius from \cite[Remark 2.4.1]{Hartshorne_Algebraic_Geometry} (composed $e$ times), and $g$ is a finite universal homeomorphism.
	
	By \autoref{invariance_univ_homeo}, it is then enough to show that $\bQ_p$--GR (resp. Witt--GR in dimension $\leq$ 3) holds for the composition $S_{p^e} \to Y \to X$. Since $S$ is a finite quotient, certainly $S_{p^e}$ is a finite quotient too, so by definition there exists a finite surjective morphism $T \to S_{p^e}$, with $T$ smooth. By the same argument as in the proof of \autoref{enough_to_show_in_proj_case}, we obtain that $\bQ_p$--GR (resp. Witt--GR) for $T \to S_{p^e} \to Y \to X$ implies $\bQ_p$--GR (resp. Witt--GR) for $S_{p^e} \to Y \to X$. Thus, the proposition follows from \autoref{Qp_GR} (resp. \autoref{Witt_GR_threefolds}).
%	By the same trace argument as in the proof of \autoref{top_fin_quot_is_Witt_CM} (and hence as in \autoref{enough_to_show_in_proj_case}), we obtain that $\bQ_p$--GR (resp. Witt--GR) for $T \to S_{p^e} \to Y \to X$ implies $\bQ_p$--GR (resp. Witt--GR) for $S_{p^e} \to Y \to X$. Thus, the proposition follows from \autoref{Qp_GR} (resp. \autoref{Witt_GR_threefolds}). \JB{can be helped!}
\end{proof}

\begin{lem}\label{traces_and_stuff}
	Let $f \colon S \to T$ be an alteration between smooth varieties. Then for all $n \geq 1$, there exist maps $W_n\omega_T \to Rf_*W_n\omega_S$ in $D(\Coh_{W_nS}^C)$ such that the composing with the trace $W_n\omega_T\to Rf_*W_n\omega_S \to W_n\omega_T$ is multiplication by $\deg(f)$. Furthermore, these maps commute with inclusions $W_n\omega \inj W_{n + 1}\omega$ and restrictions $W_{n + 1}\omega \surj W_n\omega$.
\end{lem}
\begin{proof}
	Since the sheaves $W_n\omega$ are the last piece of the $n$--truncated de Rham--Witt complex, there are canonical pullback maps $f^* \colon W_n\omega_T \to f_*W_n\omega_S$ (namely we pullback Witt differential forms as in \cite[Equation I.1.12.2]{Illusie_Complexe_de_de_Rham_Witt_et_cohomologie_cristalline}). We will show that composing them with the natural map $f_*W_n\omega_X \to Rf_*W_n\omega_X$ gives a system satisfying the assumptions of the lemma.
	
	To obtain all the commutativity properties, it is enough to show that $f^* \colon W_n\omega_T \to f_*W_n\omega_S$ commutes with $C$, $R$ and $\underline{p}$ (see \autoref{link_with_usual_de_Rham_Witt_cpx_theory}). For $R$, this follows by the very construction of the de Rham--Witt complex (see \cite[I.1.8]{Illusie_Complexe_de_de_Rham_Witt_et_cohomologie_cristalline}). Given that $p = \underline{p} \circ R$, that $f^*$ commutes with both $R$ and $p$, and that $R$ is surjective, we also deduce that $f^*$ commutes with $\underline{p}$. Finally, since $R = C \circ F$, $f^*$ commutes with both $R$ and $F$ (see \emph{loc. cit.}) and $F$ is surjective, we deduce that $f^*$ also commutes with $C$.
	
	The only thing left to show is that the composition $W_n\omega_T \to Rf_*W_n\omega_S \to W_n\omega_T$ is given by multiplication by $\deg(f)$. We may assume that $f$ is finite, and hence it follows from \cite[Proposition II.4.2.9]{Gros_Classes_de_Chern_et_classes_de_cycles_en_coho_de_hodge_witt_log} (their trace map agrees with ours by definition, see II.1.2.1 in \emph{loc. cit.}).
\end{proof}

\begin{lem}[{\cite{Rulling_Chatzimatiaou_Hodge_Witt_cohomology_and_Witt_rational_singularities}}]\label{enough_to_check_on_one_alteration}
	Let $X$ be an irreducible variety. Then if $X$ satisfies  the condition for $\bQ_p$--pseudo--rationality (resp. BE--$\bQ_p$--rationality, Witt--pseudo--rationality) for a single alteration $\pi \colon Y \to X$ with $Y$ a topologically finite quotient, then $X$ has $\bQ_p$--pseudo--rational singularities (resp. BE--$\bQ_p$--rational singularities, Witt--pseudo--rational singularities).
\end{lem}
\begin{proof}
	Let us start with BE--$\bQ_p$--rational singularities. Note that surely, if we have maps $T \to S \to X$ where both $T \to S$ and $S \to X$ are alterations, then the condition for $BE$--$\bQ_p$--rationality for $T \to X$ implies that for $S \to X$. Hence, it is enough to show that for any alteration $f \colon T \to S$ with $S$ a topologically finite quotient, the map $\bQ_{p, S} \to R\pi_*\bQ_{p, T}$ splits. Since universal homeomorphisms induce equivalence of étale sites \stacksproj{03SI}, we may assume that $S$ is a finite quotient. Consider a diagram 
	\[ \begin{tikzcd}
		T \arrow[r, "f"]                    & S                     \\
		T' \arrow[u, "\pi"] \arrow[r, "g"'] & S', \arrow[u, "\psi"']
	\end{tikzcd} \] where $T'$ and $S'$ are smooth, $g$ and $\pi$ are alterations, and $\psi$ is finite. It is then enough to show that both $\bQ_{p, S} \to R\psi_*\bQ_{p, S'}$ and $\bQ_{p, S'} \to Rg_*\bQ_{p, T'}$ split. Since the inverse system of trace maps $R\psi_*W_n\cO_{S'} = \psi_*W_n\cO_{S'} \to W_n\cO_S$ are Frobenius--equivariant, we obtain by \cite[Theorem 3.2.3]{Baudin_Duality_between_Witt_Cartier_crystals_and_perverse_sheaves} (see also \cite[Theorem 9.6.1]{Bhatt_Lurie_RH_corr_pos_char}) an inverse system of maps $R\psi_*(\bZ/p^n\bZ)_{S'} \to (\bZ/p^n\bZ)_{S}$ such that composing with $(\bZ/p^n\bZ)_{S} \to R\psi_*(\bZ/p^n\bZ)_{S'}$ gives multiplication by an integer. Taking inverse limits and inverting $p$ shows that $\bQ_{p, S} \to R\psi_*\bQ_{p, S'}$ splits. 

	We will apply a similar idea for $g$: by \autoref{traces_and_stuff}, we have a direct system of maps $W_n\omega_{S'} \to Rg_*W_n\omega_{T'}$ such that composition with the trace map $Rg_*W_n\omega_{T'} \to W_n\omega_{S'}$ is multiplication by an integer. By \cite[Theorem 5.2.7]{Baudin_Duality_between_Witt_Cartier_crystals_and_perverse_sheaves}, we obtain again an inverse system of maps $Rg_*(\bZ/p^n\bZ)_{S'} \to (\bZ/p^n\bZ)_{T'}$ such that composing with $\bZ/p^n\bZ_{T'} \to Rg_*(\bZ/p^n\bZ_{S'})$ gives multiplication by an integer. We then conclude exactly as in the previous paragraph. \\
	
	Now, let us move to the case of $\bQ_p$--pseudo--rational singularities and Witt--pseudo--rational singularities. As above, it is enough to show that for any alteration $f \colon T \to S$ between topologically finite quotients, the natural trace map $f_*W_n\omega_T \to W_n\omega_S$ is surjective up to uniform $p$--power torsion. By \autoref{invariance_univ_homeo}, we may assume that $S$ is a finite quotient. Consider a diagram as in the previous part of the proof. By \autoref{split_trace_separable} and \autoref{invariance_univ_homeo}, the maps $\psi_*W_n\omega_{S'} \to W_n\omega_S$ are surjective up to uniform $p$--power torsion, so it is enough to show that $g_*W_n\omega_{T'} \to W_n\omega_{S'}$ is surjective up to uniform $p$--power torsion. This follows from \autoref{traces_and_stuff}.
\end{proof}

Note that given a Grothendieck category $\cA$, the same argument as in the proof of \cite[Proposition 3.7.4]{Rulling_Chatzimatiaou_Hodge_Witt_cohomology_and_Witt_rational_singularities}) shows that this $D(\cA_{\bQ})$ is a $\bQ$--linear category.

\begin{lem}\label{derived_cat_nonsense}	
	Let $\cA$ be a Grothendieck category, and let $\cM^{\bullet}, \cN^{\bullet} \in D^b(\cA)$. Then the natural map \[ \Hom_{D(\cA)}(\cM^{\bullet}, \cN^{\bullet}) \otimes \bQ \to \Hom_{D(\cA_{\bQ})}(\cM^{\bullet}_{\bQ}, \cN^{\bullet}_{\bQ}) \] is an isomorphism. In particular, this holds for $\cA = \Mod(\bZ_{p, X})$ or $\cA = \Mod(W\cO_X)$.
\end{lem}
\begin{proof}
	Throughout, we will implicitly use that $(\cdot)_{\bQ}$ is exact (see \stacksproj{02MS}). Let $\cM, \cN \in \cA$ for now, we will show that the natural map $\Ext^i_{\cA}(\cM, \cN) \otimes \bQ \to \Ext^i_{\cA_{\bQ}}(\cM_{\bQ}, \cN_{\bQ})$ is an isomorphism. If $i = 0$, then this is \cite[Proposition 3.7.4]{Rulling_Chatzimatiaou_Hodge_Witt_cohomology_and_Witt_rational_singularities}. In general, note that if $\cI$ is an injective object in $\cA$, then $\cI_{\bQ}$ is again injective. Indeed, $\Hom_{\cA_{\bQ}}(-, \cI_{\bQ}) = \Hom_{\cA}(-, \cI) \otimes \bQ$, which is exact. In particular, the case of any $i \geq 0$ follows immediately from the case when $i = 0$.
	
	Now, the result for general bounded complexes follows from hypercohomology spectral sequences arguments (or equivalently, one can truncate both complexes further and further via the exact triangles in \stacksproj{08J5}, so that the result follows from the case of objects supported in a single degree).
\end{proof}

We now have all the tools to prove the main theorem of this section.

\begin{proof}[Proof of \autoref{thm_notions_of_singularities}]
	Throughout, we fix a quasi--resolution $\pi \colon Y \to X$. Let us start with the second statement (and hence we are in the case where $d \coloneqq \dim(X) \leq 3$). 
	
	The equivalence between $W\cO$--rationality and Witt--rationality is an immediate consequence of \autoref{GR_for_top_finite_quot}. Let us now show that Witt--pseudo--rationality implies $W\cO$--rationality. By definition, the maps $\pi_*W_n\omega_Y \to W_n\omega_X$ are surjective up to uniform $p$--power torsion. Note that they are also injective up to uniform $p$--power torsion, since the finite part of $Y \to X$ is purely inseparable. Furthermore, the maps $\pi_*W_n\omega_Y \to R\pi_*W_n\omega_Y$ are isomorphisms up to uniform $p$--power torsion by \autoref{GR_for_top_finite_quot}. Given that both $Y$ and $X$ are Witt--CM (see \autoref{top_fin_quot_is_Witt_CM}), we have proven that $R\pi_*W_n\omega_Y^{\bullet} \to W_n\omega_X^{\bullet}$ are isomorphisms up to uniform $p$--power torsion. By Grothendieck duality, this is also the case of $W_n\cO_X \to R\pi_*W_n\cO_Y$, so $X$ has $W\cO$--rational singularities by \autoref{Patakfalvi_Zdanowicz_criterion} and \autoref{enough_to_verify_easy_version}.
	
	Now, let us show that $W\cO$--rational singularities imply Witt--pseudo--rational singularities. By \autoref{enough_to_check_on_one_alteration}, it is enough to check the hypothesis for $\pi \colon Y \to X$. By $W\cO$--rationality and \autoref{Patakfalvi_Zdanowicz_criterion}, the cones of $W_n\cO_X \to R\pi_*W_n\cO_Y$ are of uniform $p$--power torsion. Applying Grothendieck duality gives us that the cones of \[ R\pi_*W_n\omega_Y^{\bullet} \to W_n\omega_X^{\bullet} \] are also of uniform $p$--power torsion. Since $Y$ is Witt--CM (\autoref{top_fin_quot_is_Witt_CM}), we automatically deduce by Witt--GR vanishing (\autoref{Witt_GR_threefolds}) that $X$ is Witt--CM and that the maps $\pi_*W_n\omega_Y \to W_n\omega_X$ are surjective up to uniform $p$--power torsion. \\
	
	Now, let us show the $\bQ_p$--versions of these statements. First, note that the very same argument as in the proof of \cite[Proposition 3.10]{Patakfalvi_Zdanowicz_Ordinary_varieties_with_trivial_canonical_bundle_are_not_uniruled} shows that $\bQ_p$--rationality of $X$ is equivalent to requiring the fact that the cones of $(\bZ/p^n\bZ)_X \to R\pi_*(\bZ/p^n\bZ)_Y$ are of uniform $p$--power torsion. By the compatibilities in the Riemann--Hilbert correspondence \cite[Theorem 3.2.3]{Baudin_Duality_between_Witt_Cartier_crystals_and_perverse_sheaves} (see also the original reference \cite[Section 9]{Bhatt_Lurie_RH_corr_pos_char}), this is in turn equivalent to the fact that the cones of $W_n\cO_X \to R\pi_*W_n\cO_Y$ are nilpotent up to uniform $p$--power torsion. With this in mind, the proof that having $\bQ_p$--pseudo--rationality is equivalent to $\bQ_p$--rationality is identical to the Witt vector case, replacing uses of Grothendieck duality by \cite[Theorem 5.2.7]{Baudin_Duality_between_Witt_Cartier_crystals_and_perverse_sheaves}.
	
	The fact that $\bQ_p$--rationality implies $BE$--$\bQ_p$--rationality is immediate: since the hypothesis of $BE$--$\bQ_p$--rationality is trivially satisfied on $\pi$, we are done by \autoref{enough_to_check_on_one_alteration}. Let us now conclude this proof by showing that $BE$--$\bQ_p$--rationality implies $\bQ_p$--pseudo--rationality, so assume the existence of a splitting $s \colon R\pi_*\bQ_{p, Y} \to \bQ_{p, X}$. By \autoref{derived_cat_nonsense}, there exists $g \colon R\pi_*\bZ_{p, Y} \to \bZ_{p, X}$ such that $r \cdot g_{\bQ} = s$ for some $r \in \bQ$. Consider the composition $\bZ_{p, X} \to R\pi_*\bZ_{p, Y} \xto{g} \bZ_{p, X}$. This corresponds to a global section $t \in \bZ_{p, X}$, such that $rt_{\bQ} = 1$. In particular, $t$ is automatically an element of $\bZ_{(p)}$. Applying $- \otimes^L_{\bZ_p} \bZ/p^n\bZ$  gives commutative diagrams \[ \begin{tikzcd}
		(\bZ/p^n\bZ)_X \arrow[rr, "nat"] \arrow[rrrr, "t"', bend right = 15] &  & R\pi_*(\bZ/p^n\bZ)_Y \arrow[rr, "g_n"] &  & (\bZ/p^n\bZ)_X.
	\end{tikzcd}\] Applying \cite[Theorem 5.2.7]{Baudin_Duality_between_Witt_Cartier_crystals_and_perverse_sheaves} and perfection (\autoref{Gabber_finiteness}), we obtain commutative diagrams 
	\[ \begin{tikzcd}
		W_n\omega_X^{\bullet, \perf} \arrow[rr] \arrow[rrrr, "t"', bend 	right = 15] &  & R\pi_*W_n\omega_Y^{\bullet, \perf} \arrow[rr, "\Tr"] &  & W_n\omega_X^{\bullet, \perf}.
	\end{tikzcd}\] Since $Y$ is Witt--CM by \autoref{top_fin_quot_is_Witt_CM}, we deduce from \autoref{GR_for_top_finite_quot} that $X$ is also Witt--CM. Furthermore, applying $\cH^{-\dim X}$ also shows that the maps $\pi_*W_n\omega_Y^{\perf} \to W_n\omega_X^{\perf}$ are surjective up to uniform $p$--power torsion. The proof is therefore complete, thanks to \autoref{enough_to_check_on_one_alteration}.
\end{proof}

\subsection{Witt--Cohen--Macaulayness and rationality of certain singularities}

\begin{cor}
	Let $X$ be a threefold of klt type over $k$. Then $X$ has Witt-CM singularities.
\end{cor}
\begin{proof}
	If $X$ is $\bQ$--factorial, then it has $W\cO$--rational singularities by \cite[Corollary 1.3]{Hacon_Witaszek_On_the_relative_MMP_for_threefolds_in_low_char}. In fact, the $\bQ$--factoriality assumption is not needed, since we can use the non--$\bQ$--factorial MMP from \cite{Kollar_Relative_MMP_without_Q_factoriality}. We then conclude the proof by \autoref{thm_notions_of_singularities}.
\end{proof}

To make sense of the second application to singularities, we will need the following definitions:

\begin{defn}\label{def_F_p_notions}
	Let $X$ be a variety.
	\begin{itemize}
		\item We say that $X$ has \emph{$\bF_p$--Cohen--Macaulay} (in short, $\bF_p$--CM) singularities if for all $i \neq \dim X$, $\cH^i(\omega_X^{\bullet})^{\perf} = 0$\footnote{In \cite{Baudin_Bernasconi_Kawakami_Frobenius_GR_fails}, we named this notion \emph{Cohen--Macaulayness up to nilpotence}. We decided to introduce this new name here because it fits better with the terminologies for $\bQ_p$--Cohen--Macaulayness and Witt--Cohen--Macaulayness}.
		\item If $X$ is irreducible, then we say that $X$ has \emph{$\bF_p$--rational singularities} if there exists a resolution of singularities $\pi \colon Y \to X$ such that for all $i > 0$, $R^i\pi_*\bF_{p, Y} = 0$.
	\end{itemize}
\end{defn}

\begin{prop}\label{nice_vanishing_Witt_ps_rat}
	Let $X$ be a variety with $\bQ_p$--rational singularities, and let $\pi \colon Y \to X$ be a quasi--resolution. Then $\pi_*W\omega_Y = W\omega_X$.
	
	If in addition $X$ has $\bF_p$--CM singularities, then also $R^1\pi_*W\omega_Y = 0$ and $\pi_*\omega_Y^{\perf} = \omega_X^{\perf}$.
\end{prop}
\begin{proof}
	Since $X$ has $\bQ_p$--pseudo--rational singularities by \autoref{thm_notions_of_singularities}, we deduce that the cokernel $\cC$ of the trace map $\pi_*W\omega_Y \to W\omega_X$ is of $p^e$--torsion for some $e > 0$. Note that this trace map is actually injective. Indeed, let 
	\[ \begin{tikzcd}
		Y \arrow[r, "f"] \arrow[rr, "\pi"', bend right] & Z \arrow[r, "g"] & X
	\end{tikzcd} \] denote the Stein factorization of $\pi$. Surely, $f_*W\omega_Y \to W\omega_Z$ is injective (since $f_*\omega_Y \to \omega_Z$ is). Furthermore, since $g$ is purely inseparable, there is a factorization 
	\[ \begin{tikzcd}
		X \arrow[r, "h"] \arrow[rr, "F^s"', bend right] & Z \arrow[r, "g"] & X,
	\end{tikzcd} \] so we also have a factorization $F^s_*g_*W\omega_Z \to F^s_*W\omega_X \to g_*W\omega_Z \to W\omega_X$. Since both the trace maps of $W\omega_X$ and $W\omega_Z$ are isomorphisms by \autoref{Witt_omega_coperfect}, we deduce that $g_*W\omega_Z = W\omega_X$. In particular, we have proven our sought injectivity, so we have a short exact sequence \[\begin{tikzcd}
	0 \arrow[rr] &  & \pi_*W\omega_Y \arrow[rr, "\Tr"] &  & W\omega_X \arrow[rr] &  & \cC \arrow[rr] &  & 0.
	\end{tikzcd} \] Applying $- \otimes^L_{\bZ} \bF_p$ gives an exact sequence \[ 0 \to \pi_*W\omega_Y[p] \to W\omega_X[p] \to \cC[p] \to (\pi_*W\omega_Y)/p  \to (W\omega_X)/p \to \cC/p \to 0. \]
	By \autoref{Witt_omega_p_torsion_free}, both $\pi_*W\omega_Y$ and $W\omega_X$ are $p$--torsion--free, so the first two terms of this sequence vanish. Using the sequence from \autoref{cone_of_p_Witt_omega} on both $X$ and $Y$, we obtain a commutative diagram 
	\begin{equation}\label{cool_square}
		\begin{tikzcd}
			(\pi_*W\omega_Y)/p \arrow[r] \arrow[d, hook] & W\omega_X/p \arrow[d, hook] \\
			\pi_*\omega_Y^{\perf} \arrow[r]              & \omega_X^{\perf}.
		\end{tikzcd} 
	\end{equation} 
	Since the trace map $\pi_*\omega_Y \to \omega_{X_{\red}}$ is injective and $\omega_{X_{\red}} \to \omega_X$ is an isomorphism up to nilpotence, we deduce that $\pi_*\omega_Y^{\perf} \to \omega_X^{\perf}$ is in fact injective. The top arrow in the diagram above must then be injective too, so $\cC[p] = 0$. Since $\cC$ is $p^e$--torsion, this forces $\cC = 0$, so $\pi_*W\omega_Y = W\omega_X$. \\
	
	Now, assume furthermore that $X$ has $\bF_p$--CM singularities. Since $\omega_X^{\bullet, \perf} = \omega_X^{\perf}[\dim X]$ by assumption, we deduce from \autoref{Witt_omega_coperfect} and \autoref{Witt_omega_Mittag_Leffler} that $W\omega_X^{\bullet} = W\omega_X[\dim X]$. Then $W\omega_X/p = \omega_X^{\perf}$ by \autoref{cone_of_p_Witt_omega}, so given that $\pi_*W\omega_Y = W\omega_X$, we deduce by taking modulo $p$ and \autoref{cool_square} that $\pi_*\omega_Y^{\perf} = \omega_X^{\perf}$. We are left to show the vanishing $R^1\pi_*W\omega_Y = 0$. We will play a similar game as in the first part of the proof: consider the exact triangle 
	\[ \begin{tikzcd}
		R\pi_*W\omega_Y^{\bullet} \arrow[rr, "\Tr"] &  & W\omega_X^{\bullet} \arrow[rr] &  & \cT \arrow[rr, "+1"] &  & {}
	\end{tikzcd} \]
	By $\bQ_p$--GR vanishing (\autoref{GR_for_top_finite_quot}) and $\bQ_p$--Cohen--Macaulayness of both $X$ and $Y$ (see \autoref{top_fin_quot_is_Witt_CM}), we deduce that $\cT$ is $p^e$--torsion for some $e > 0$.	Given that $\pi_*W\omega_Y = W\omega_X$ and that $W\omega_X^{\bullet}$ is concentrated in degree $-\dim X$, we deduce that $\cT \in D^{\geq - \dim X}$ and that $\cH^{-\dim X}(\cT) = \cH^{- \dim X  + 1}(R\pi_*W\omega_Y^{\bullet})$. Given that $R^1\pi_*W\omega_Y \inj \cH^{- \dim X  + 1}(R\pi_*W\omega_Y^{\bullet})$, it is enough to show that $\cH^{-\dim X}(\cT) = 0$ in order to obtain the vanishing we want. 
	
	Taking $- \otimes^L_{\bZ} \bF_p$ gives the exact triangle 
	\[ \begin{tikzcd}
		{R\pi_*\omega_Y^{\bullet, \perf}} \arrow[rr, "\Tr"] &  & {\omega_X^{\bullet, \perf}} \arrow[rr] &  & \cT \otimes^L_{\bZ} \bF_p \arrow[rr, "+1"] &  & {}
	\end{tikzcd} \] 
	so the first few terms of the associated long exact sequence give \[ \begin{tikzcd}
		0 \arrow[rr] &  & {\cH^{-\dim X}(\cT)[p]} \arrow[rr] &  & {\pi_*\omega_Y^{\perf}} \arrow[rr, "\Tr"] &  & {\omega_X^{\perf}} \arrow[rr] &  & \dots.
	\end{tikzcd} \]
	Again, since the trace map is injective, this shows that $\cH^{-\dim X}(\cT)[p] = 0$, so also $\cH^{-\dim X}(\cT) = 0$ since this module is $p^e$--torsion.
\end{proof}

For the definition of quasi--$F$--rational singularities and its properties, we refer the reader to \cite{Quasi_F_splittings_III}. In particular, a variety with strongly $F$--regular or $F$--rational singularities also has quasi--$F$--rational singularities.

\begin{thm}
	Let $X$ be a variety with quasi--$F$--rational singularities. Then $X$ has $\bQ_p$--rational singularities. If furthermore $\dim(X) \leq 3$, then $X$ has both Witt--rational singularities and $\bF_p$--rational singularities.
\end{thm}
\begin{proof}
	We will start by showing that $X$ has Witt--pseudo--rational singularities in any dimension. Since Witt--pseudo--rational singularities are $\bQ_p$--pseudo--rational by definition, this will show that $X$ has $\bQ_p$--rational singularities by \autoref{thm_notions_of_singularities}. This will also show that $X$ has Witt--rational singularities when $\dim(X) \leq 3$ by \emph{loc. cit.}.
	
	Since $X$ is CM by definition, it is immediate to see that it is Witt--CM. Let $\pi \colon Y \to X$ be a quasi--resolution. By \autoref{enough_to_check_on_one_alteration}, it is enough to show that the maps $\pi_*W_n\omega_Y \to W_n\omega_X$ are surjective up to uniform $p$--power torsion. Throughout, we will use notations from \cite{Quasi_F_splittings_III}.
	
	First of all, note that $X$ is integral by definition of quasi--$F$--rationality, and it is in fact normal by applying \cite[Proposition 3.11.(3)]{Quasi_F_splittings_III} to the normalization morphism. Now, let $Y \xto{g} Z \xto{f} X$ denote the Stein factorization of $\pi$. Then each map $g_*W_n\omega_Y \to W_n\omega_Z$ is injective. By normality of $X$ and the fact that $Z \to X$ is purely inseparable, we deduce that for some $e > 0$, $\cO_Z^{p^e} \inc \cO_X$. This shows that the kernel of $f_*W_n\omega_Z \to W_n\omega_X$ is killed by the $e$'th power of the Cartier operator. Combining both observations, we deduce that the kernel of $\pi_*W_n\omega_Y \to W_n\omega_X$ is $p^e$--torsion for any $n \geq 1$.
	
	Now, fix $n \geq 1$, and let $s \in W_n\omega_X$. Since \[ \colim \tau(W_n\omega_X) = \colim W_n\omega_X \] by \cite[Proposition 3.34.(2)]{Quasi_F_splittings_III}, there exists $m \geq n$ and $t \in \pi_*W_m\omega_Y$ which is sent to the image of $s$ in $W_m\omega_X$ (recall that $W_n\omega_X \inc W_m\omega_X$).  Consider the diagram 
	\[ \begin{tikzcd}
		0 \arrow[r] & \pi_*W_n\omega_Y \arrow[r] \arrow[d] & \pi_*W_m\omega_Y \arrow[r] \arrow[d] & F^{n - 1}_*\pi_*W_{m - n}\omega_Y \arrow[d] \\
		0 \arrow[r] & W_n\omega_X \arrow[r]                & W_m\omega_X \arrow[r]                & F^{n - 1}_*W_{m - n}\omega_X               
	\end{tikzcd}\] with exact rows. Then $t$ is sent to the kernel of $F^{n - 1}_*\pi_*W_{m - n}\omega_Y \to F^{n - 1}_*W_{n - m}\omega_X$, so by the previous paragraph, $p^et \in \pi_*W_n\omega_Y$, whence $p^es \in \im(\pi_*W_n\omega_Y \to W_n\omega_X)$. In other words, we have proven that the cokernel of $\pi_*W_n\omega_Y \to W_n\omega_X$ is $p^e$--torsion, so $X$ indeed has Witt--pseudo--rational singularities.
	
	We are now left to show the $\bF_p$--rationality statement, so assume that $\dim(X) = 3$, and let $\pi \colon Y \to X$ be a resolution of singularities. By \autoref{nice_vanishing_Witt_ps_rat}, we know that $R^1\pi_*W\omega_Y = 0$. Since \autoref{cone_of_p_Witt_omega} gives a short exact sequence \[ 0 \to W\omega_Y \to W\omega_Y \to \omega_Y^{\perf} \to 0 \] and $R^2\pi_*\omega_Y = 0$ by \autoref{GR_surfaces}, we deduce that $R^1\pi_*\omega_Y^{\perf} = 0$. Since $\pi_*\omega_Y^{\perf} = \omega_X^{\perf}$ (see \autoref{nice_vanishing_Witt_ps_rat}), we deduce that the trace map $R\pi_*\omega_Y \to \omega_X$ is an isomorphism in $D^b(\Crys_X^C)$. By Grothendieck duality for crystals (\cite[Theorem 5.2.7]{Baudin_Duality_between_perverse_sheaves_and_Cartier_crystals}), we conclude that the map $\bF_{p, X} \to R\pi_*\bF_{p, Y}$ is an isomorphism. In other words, $X$ has $\bF_p$--rational singularities.
\end{proof}

\bibliographystyle{alpha}
\bibliography{bibliography}

\Addresses

\end{document}